\documentclass[12pt]{amsart}
\usepackage[T1]{fontenc}
\theoremstyle{remark}
\usepackage{units}
\usepackage{mathrsfs}
\usepackage{amsbsy}
\usepackage{amstext}
\UseRawInputEncoding
\usepackage{xcolor}
\usepackage[english]{babel}
\usepackage{calc}
\numberwithin{equation}{section}
\numberwithin{figure}{section}
\theoremstyle{plain}
\newtheorem{thm}{\protect\theoremname}[section]
\theoremstyle{definition}
\newtheorem{defn}[thm]{\protect\definitionname}
\usepackage{amsthm}
\usepackage{amssymb}
\usepackage[a4paper,margin=20mm]{geometry}
\usepackage{bm}

\makeatletter

\newcommand{\lyxmathsym}[1]{\ifmmode\begingroup\def\b@ld{bold}
  \text{\ifx\math@version\b@ld\bfseries\fi#1}\endgroup\else#1\fi}

\newtheorem{rem}[thm]{\protect\remarkname}
\theoremstyle{plain}
\newtheorem{lem}[thm]{\protect\lemmaname}
\theoremstyle{plain}
\newtheorem{prop}[thm]{\protect\propositionname}
\theoremstyle{plain}

\theoremstyle{remark}
\newtheorem*{rem*}{\protect\remarkname}

\usepackage{erewhon} 

\usepackage{amsmath,amssymb,amscd,amsthm,amsfonts}
\usepackage{newtxmath}     

\usepackage[T1]{fontenc}
\usepackage[utf8]{inputenc}

\usepackage{graphicx}
\usepackage{mathtools}
\usepackage{url}
\usepackage{enumerate}
\usepackage{multirow,bm}
\usepackage[all,color]{xy}        
\usepackage{tikz-cd}
\usepackage{color,xcolor}
\usepackage[normalem]{ulem}
\usepackage{verbatim}
\usepackage{colortbl}
\usepackage{float}
\usepackage{amsmath}
\allowdisplaybreaks

\usepackage{chngcntr}

\theoremstyle{definition}

\numberwithin{equation}{section}

\newcommand{\cyr}{%
  \renewcommand\rmdefault{wncyr}%
  \renewcommand\sfdefault{wncyss}%
  \renewcommand\encodingdefault{OT2}%
  \normalfont\selectfont
}
\DeclareTextFontCommand{\textcyr}{\cyr}

\makeatletter
\def\@settitle{%
  \begin{center}
    \baselineskip14\p@\relax
    \LARGE\@title
  \end{center}}
\makeatother

\newcommand{\sN}{\mathcal{N}}
\newcommand{\sO}{\mathcal{O}}

\newcommand{\mA}{\mathbb{A}}
\newcommand{\mC}{\mathbb{C}}

\newcommand{\mN}{\mathbb{N}}
\newcommand{\mP}{\mathbb{P}}
\newcommand{\mQ}{\mathbb{Q}}

\newcommand{\rank}{\mathrm{rank}\,}

\newcounter{myparagraph}[subsection]

\author{Hiromichi Takagi}
\address{Department of Mathematics, Gakushuin University, 
	Mejiro, Toshima-ku, Tokyo 171-8588, Japan}
\email{hiromici@math.gakushuin.ac.jp}

\makeatother

\providecommand{\corollaryname}{Corollary}
\providecommand{\definitionname}{Definition}
\providecommand{\lemmaname}{Lemma}
\providecommand{\propositionname}{Proposition}
\providecommand{\remarkname}{Remark}
\providecommand{\theoremname}{Theorem}

\begin{document}
\title{\textbf{Key variety construction of Sarkisov links for prime $\mQ$-Fano threefolds
of codimension four associated to Type ${\rm II}_{2}$ projections}}
\begin{abstract}
In our paper \cite{Tak6}, we constructed eight families of quasi-smooth
prime $\mathbb{Q}$-Fano threefolds, anticanonically embedded in codimension
four, using weighted projectivizations of the $14$-dimensional affine
variety $\Pi_{\mathbb{A}}^{14}$or its cone. Let $\widehat{f}\colon\widehat{X}\to X$
be the unique divisorial extraction at one specified singularity of
maximal index. In this paper, we explicitly construct the Sarkisov
link starting from $\widehat{f}$ for $X$ belonging to seven of these
families. This is achieved by using the Sarkisov link associated with
the weighted projectivization of $\Pi_{\mathbb{A}}^{14}$ or its cone
corresponding to $X$. As a consequence, we show that the Sarkisov
link ends with either a fibration whose general fiber is a del Pezzo
surface of degree one or a divisorial contraction of type $(2,1)$
to weighted complete intersections of codimension at most two. We
also provide more detailed descriptions of these Sarkisov links.
\end{abstract}

\maketitle
2020\textit{ Mathematics subject classification}: 14J45, 14E30


\markboth{Sarkisov links for $\mQ$-Fano 3-folds via a key variety}{Hiromichi
Takagi}

\tableofcontents

\section{\textbf{Introduction}}

\subsection{Background}

A normal complex projective variety with at most terminal singularities and with ample anticanonical divisor is called a \textit{Fano variety}. 
In this paper, we restrict ourselves to $3$-dimensional Fano varieties (Fano $3$-folds) that are quasi-smooth and whose anticanonical divisor generates the group of numerical equivalence classes of Weil divisors. 
Such varieties will simply be called \textit{prime Fano $3$-folds} in this paper (cf.\ \cite{Mu}).

The classification and birational geometry of prime $\mQ$-Fano $3$-folds have been studied extensively. 
The present work is part of this line of research, and one of its main sources is \cite{BZ}, in which  the authors study prime $\mQ$-Fano $3$-folds defined by the Pfaffians of the five principal $4 \times 4$ minors of a $5 \times 5$ alternating matrix. 
Each of these varieties has codimension $3$ in the ambient weighted projective space and admits a birational map called a Type I projection. 
The Sarkisov links for these varieties associated with Type I projections are classified from the viewpoint of toric Sarkisov links in the ambient weighted projective spaces. 
This work is also important as it establishes a foundation for applying toric Sarkisov links to explicit birational geometry.

Subsequently, in \cite{Ca}, a similar approach was used to study prime $\mQ$-Fano $3$-folds of codimension $4$ in the ambient weighted projective spaces that admit so-called Tom Type I projections, and the corresponding Sarkisov links were classified (see also \cite{Tak1}). 
We also note here that in \cite{Tak8}, prime $\mQ$-Fano $3$-folds admitting Type I projections that are not of  Tom type are constructed using Sarkisov links.

In this paper, we consider seven families of prime $\mQ$-Fano $3$-folds of codimension $4$ in the ambient weighted projective spaces that admit birational maps called Type ${\rm II}_{2}$ projections in \cite{Tay}, and we classify the associated Sarkisov links from the viewpoint of toric Sarkisov links.

We briefly explain how the prime $\mQ$-Fano $3$-folds studied in this paper fit into the broader picture of codimension $4$ prime $\mQ$-Fano $3$-folds. 
According to the online database \cite{GRDB}, there are $145$ possible Hilbert series for codimension $4$ prime $\mQ$-Fano $3$-folds. 
Among them, $143$ are expected to have defining ideals generated by $9$ elements with $16$ linearly independent relations among them. 
Moreover, for each of these $143$ cases, it is expected that there exist two topologically distinct families. 
In the case of Type I projections, these correspond to the two types called Tom and Jerry in \cite{BKR}. 
From the viewpoint of key varieties, these correspond to $\mP^{2}\times\mP^{2}$-fibrations or $\mP^{1}\times\mP^{1}\times\mP^{1}$-fibrations. 
From the viewpoint of algebraic structures, they correspond to quadratic Jordan algebras of cubic forms or Freudenthal triple systems. 
Although these three viewpoints do not correspond perfectly, they are closely related. 
The classification of codimension $4$ prime $\mQ$-Fano $3$-folds is still incomplete, but from the perspective of constructing examples, substantial progress has been made along these lines.

The prime $\mQ$-Fano $3$-folds studied in this paper were constructed in \cite{Tak6}. 
They are obtained as weighted complete intersections in the weighted projectivizations of the affine variety $\Pi_{\mA}^{14}$ or in its cone obtained by adding a free coordinate of weight $1$. The affine variety $\Pi_{\mA}^{14}$ was constructed in two ways; by Type $\rm{II}_2$-unprojection in \cite{Tak2}, and by quadratic Jordan algebra of cubic form in \cite{Tak5}. It is related with $\mP^2\times \mP^2$-fibration.
For a prime $\mQ$-Fano $3$-fold $X$, we call the weighted projectivization of $\Pi_{\mA}^{14}$ or its cone that realizes $X$ as a weighted complete intersection the \textit{key variety} of $X$. 
In this way, eight families of prime $\mQ$-Fano $3$-folds are obtained, but one of them is shown to be birationally superrigid in \cite{O}. 
Therefore, in this paper we treat the remaining seven families.

Since these prime $\mQ$-Fano $3$-folds are constructed via key varieties, their Sarkisov links can be obtained by first constructing the Sarkisov links of the key varieties using toric Sarkisov links, and then restricting them to $3$-folds. 
This viewpoint was developed in \cite{Tak3, Tak4, Tak8}, and we also follow it in the present paper, which allows us to treat the construction of Sarkisov links of prime $\mQ$-Fano $3$-folds in a more unified way.

\subsection{Main results}

In this subsection, we state the main results of this paper. 
Let $\Pi_{\mP}^{13}$ or $\Pi_{\mP}^{14}$ denote the weighted projectivization of $\Pi_{\mA}^{14}$ or its cone, respectively, which realize a prime $\mQ$-Fano $3$-fold $X$ considered in this paper as a weighted complete intersection. 
When we treat them uniformly, we denote them collectively by $\Pi_{\mP}$. Then $X$ satisfies one of the following conditions \eqref{eq:TypeDp} or \eqref{eq:TypeDiv} (these are restatements of the results of \cite{Tak6}, with type names assigned):

\begin{align}
\textbf{\text{Type\,Dp}}: \quad 
&  X=\Pi_{\mP}^{14}\cap(2)^{2}\cap(3)\cap(4)^{2}\cap(5)\cap(6)^{2}\cap(d-3)\cap(d-2)\cap(d-1) \label{eq:TypeDp}\\
& \hspace{2em} (4\le d\le 7) \notag\\
\textbf{\text{Type\,Div}}: \quad 
&  X=\Pi_{\mP}^{13}\cap(2)^{3}\cap(3)^{3}\cap(4)\cap(d-1)\cap(d)\cap(d+1) \label{eq:TypeDiv}\\
& \hspace{2em} (2\le d\le 4) \notag
\end{align}
In terms of the numbering in \cite{GRDB}, the cases of Type~Dp with $d=7,6,5,4$ correspond to Nos.~501, 512, 550, and 872, respectively, while the cases of Type~Div with $d=4,3,2$ correspond to Nos.~577, 878, and 1766, respectively. 
In the case of No.~872 ($d=4$), the variety $X$ is constructed in \cite{Tak6} as a weighted complete intersection in the weighted projectivization $\Pi_{\mP}^{13}$ of $\Pi_{\mA}^{14}$. 
However, in this paper, we instead consider its cone $\Pi_{\mP}^{14}$ and describe $X$ as a weighted complete intersection in $\Pi_{\mP}^{14}$ by cutting it with a hypersurface of weight $1$ that does not pass through the vertex of the cone. 
This allows us to treat No.~872 uniformly as a case of Type~Dp.

The variety $\Pi_{\mP}$ has a coordinate $s_{1}$ (to be described later), and it has the following singularity at the $s_{1}$-point:
\begin{equation}
\begin{cases}
\frac{1}{d+1}(1^{2},2^{2},3,4^{2},5,6^{2},d-3,d-2,d-1,d)\text{-singularity}: &  \Pi_{\mP}=\Pi_{\mP}^{14},\\
\frac{1}{d+1}(1^{2},2^{3},3^{3},4,d-1,d^{2},d+1)\text{-singularity}: &  \Pi_{\mP}=\Pi_{\mP}^{13}.
\end{cases}
\label{eq:s1-pt}
\end{equation}
Let $\widehat{f}_{\Pi}\colon \widehat{\Pi}\to \Pi_{\mP}$ denote the weighted blow-up at the $s_{1}$-point with these weights.

Under the above preparation, the Sarkisov links for $\Pi_{\mP}^{13}$ and $\Pi_{\mP}^{14}$ can be described as follows.

\begin{thm}[\textbf{Sarkisov links for key varieties}]
\label{thm:keySarkisov}

The Sarkisov link starting from $\widehat{f}_{\Pi}$ 
\begin{equation}
\label{eq:PiSarkisov}
\xymatrix{
& \widehat{\Pi} \ar@{-->}[rr] \ar[dl]_{\widehat{f}_{\Pi}} && \widetilde{\Pi} \ar[dr]^{\widetilde{f}_{\Pi}}\\
\Pi_{\mP} &&&& \mP,
}
\end{equation}
is constructed, and satisfies the following properties in the cases of \textup{Type~Dp} and \textup{Type~Div}, respectively:

\vspace{3pt}

\noindent \textbf{\textup{Type~Dp:}} 
The map $\widehat{\Pi}\dashrightarrow \widetilde{\Pi}$ is the composition of one flop and one flip. 
The flip is a toric flip of type $(1^{2},2^{2},3,4^{2},5,6^{2},d-3,d-1,d, \ -2,-1)$. 
We have $\,\mP=\mP^{1}$, and $\widetilde{f}_{\Pi}$ is a locally trivial $\,\mP(1^{2},2^{3},3^{2},4^{2},5,6,d-3,d-2,d-1)$-bundle. 
The isomorphism class of this bundle is completely determined since the transition functions over $\,\mP^{1}$ are given in Proposition~\ref{prop:transition-fcn}.

\vspace{3pt}

\noindent \textbf{\textup{Type~Div:}} 
The map $\,\widehat{\Pi}\dashrightarrow \widetilde{\Pi}$ is a flop. 
The variety $\,\mP$ is the weighted projective space $\mP(1^{4},2^{4},3^{3},d-1,d^{2})$. 
The morphism $\widetilde{f}_{\Pi}$ is a divisorial contraction whose exceptional divisor is mapped to a curve. 
The image $\,C$ of the exceptional divisor is a conic contained in $\mP(1,1,1)\simeq \mP^{2}$ inside $\mP$, and the exceptional divisor is a locally trivial $\mP(1^{2},2^{3},3^{3},4,d-1,d,d+1)$-bundle over $C$.
\vspace{3pt}

Moreover, $\,\Pi_{\mP}$ has only terminal singularities.
\end{thm}

The statement that the Sarkisov link can be constructed in Theorem~\ref{thm:keySarkisov} is itself nontrivial, and we consider that establishing this is one of the main contributions of the present paper. 
For a more detailed explanation of this point, we refer to the end of Section~\ref{subsec:Description-of Pi1}. 
We also note that the statement at the end of the theorem that $\Pi_{\mP}$ has only terminal singularities follows as a consequence of the description of the Sarkisov link.

\vspace{5pt}

By studying in detail the restriction of Theorem~\ref{thm:keySarkisov} to a $\mQ$-Fano $3$-fold $X$, we obtain the following Theorem~\ref{thm:FanoSarkisov}. 
Before stating it, we introduce the following definition.

\begin{defn}[\textbf{ordinary $cA_{r-1}$-singularity and simple $(2,1)$-contraction}]
Let $r\geq 2$ be an integer, and let $g$ be an element of order $r$ in the convergent power series ring $\mC\{z,w\}$ (here, the order of $g$ means the minimum of the orders of the monomials appearing in $g$). 
Assume that the leading term $g_{r}$ of $g$ is a product of $r$ mutually coprime linear forms. 
Then, by Hensel's lemma (cf.~\cite[4.24 (3)]{KoMo2}), the series $g$ itself can be written as a product of $r$ convergent power series whose leading terms coincide with the irreducible factors of $g_{r}$. 
Then it follows that the hypersurface $H=\{xy+g=0\}$ has an isolated singularity at the origin, which is a $cA_{r-1}$-singularity. 
A singularity analytically isomorphic to this one is called an \textit{ordinary $cA_{r-1}$-singularity}. (Note that the ordinary $cA_1$-singularity is just the ordinary double point.)

Moreover, when we assign weights $(0,r,1,1)$ to the variables $x,y,z,w$, we call the weighted blow-up $\widetilde{H}\to H$ of weight $(0,r,1,1)$,
\[
(u,X,Y,Z,W)\mapsto(X,u^{r}Y,uZ,uW),
\]
a \textit{simple $(2,1)$-contraction}. 

A straightforward computation shows the following: the exceptional locus of $\widetilde{H}\to H$ is an irreducible divisor contracted to the $x$-axis, and all fibers are one-dimensional. 
The fiber over the origin contains a unique $\nicefrac{1}{r}(1,1,r-1)$-singularity, and $\widetilde{H}$ is smooth outside this point. 
The fiber over the origin consists of $r$ irreducible components passing through this singularity.
\end{defn}

\begin{thm}[\textbf{Sarkisov links for prime $\mQ$-Fano $3$-folds}]
\label{thm:FanoSarkisov}

Let $X$ be a general weighted complete intersection in $\Pi_{\mP}$ satisfying either \eqref{eq:TypeDp} or \eqref{eq:TypeDiv} and the following condition $(\star):$

\vspace{3pt}

\noindent $(\star)$ 
In \eqref{eq:TypeDp} and \eqref{eq:TypeDiv}, the defining equations of the weighted homogeneous hypersurfaces (which we call \textit{sections}) cutting out $\,X$ from $\,\Pi_{\mP}$ do not contain the variables $\,s_{1}$, $s_{2}$, $s_{3}$.
(Note that the ambient weighted projective space of $\,\Pi_{\mP}$ has coordinates $s_{1}$, $s_{2}$, $s_{3};$ see \cite{Tak6}.)

\vspace{3pt}

Then $\,X$ is a prime $\,\mQ$-Fano $3$-fold. It contains the $s_{1}$-point of $\ \Pi_{\mP}$, which is a $\nicefrac{1}{d+1}(1,1,d)$-singularity. 
Let $\widehat{X}$ be the strict transform of $X$ in $\widehat{\Pi}$. 
Then the restriction $\widehat{f}\colon \widehat{X}\to X$ of $\widehat{f}_{\Pi}$ to $\widehat{X}$ is the weighted blow-up at the $s_{1}$-point with weights $\nicefrac{1}{d+1}(1,1,d)$.

The restriction of the Sarkisov link \eqref{eq:PiSarkisov} to $X$ and its strict transforms is the Sarkisov link starting from $\widehat{f}$. 
Denoting it by
\begin{equation}
\label{eq:3-foldSarkisov}
\xymatrix{
& \widehat{X} \ar@{-->}[rr] \ar[dl]_{\widehat{f}} && \widetilde{X} \ar[dr]^{\widetilde{f}}\\
X &&&& Y,
}
\end{equation}
we describe it as follows in the cases of \textup{Type~Dp} and \textup{Type~Div}, respectively.

\vspace{3pt}

\noindent \textbf{\textup{Type~Dp:}} 
\begin{itemize}
	\item The map $\widehat{X}\dashrightarrow \widetilde{X}$ is the composition of a flop and a flip. 
	The flip is of hypersurface type classified in \cite{B}, with $\mC^{*}$-action $(d,1^{2},-2,-1;d-2)$.
	
	\item The flipping curve has $d-2$ irreducible components. 
	Each passes through the unique $\nicefrac{1}{d}(1,1,d-1)$-singularity of $\,\widehat{X}$ on the $\widehat{f}$-exceptional divisor.
	
	\item The flipped curve is irreducible and passes through two singularities: 
	an ordinary $cA_{d-3}$-singularity and a $\nicefrac{1}{2}(1,1,1)$-singularity.
	
	\item $\,Y\simeq \mP^{1}$, and $\widetilde{f}\colon \widetilde{X}\to Y$ is a fibration whose general fiber is a del Pezzo surface of degree $1$.
\end{itemize}

\vspace{3pt}

\noindent \textbf{\textup{Type~Div:}} 
\begin{itemize}
	\item The map $\widehat{X}\dashrightarrow \widetilde{X}$ is a flop. 
	
	\item The morphism $\widetilde{f}$ is a divisorial contraction whose exceptional divisor is mapped to the curve $C$, 
	where $C$ is as described in Theorem~\ref{thm:keySarkisov}. 
	
	\item Any singularity of $\,\widetilde{Y}$ on the $\widetilde{f}$-exceptional divisor is a $\nicefrac{1}{r}(1,1,r-1)$-singularity for some $r\in\mN$. 
	Near the image of each such singularity, $\widetilde{f}$ is a simple $(2,1)$-contraction. 
	
	\item The variety $Y$ is a weighted complete intersection of two hypersurfaces of weights $4$ and $d+1$ in $\mP(1^{4},2,d)$. 
	When $d=4$, the coordinate of weight $\,4$ can be eliminated using the hypersurface of weight $\,4$. 
	In this case, $Y$ becomes a hypersurface of weight $\,5$ in $\mP(1^{4},2)$. 
	In either case, $Y$ belongs to the families of weighted complete intersection $\mQ$-Fano $\,3$-folds listed in \cite{IF}. 
	
	\item The variety $Y$ has the same cyclic quotient terminal singularities as the corresponding quasi-smooth $\mQ$-Fano $3$-fold in \cite{IF} outside $C$. 
	On $C$, it has five singularities: three ordinary $cA_{2}$-singularities, one ordinary $cA_{d-1}$-singularity, and one ordinary $cA_{d}$-singularity.
\end{itemize}
\end{thm}

\begin{rem}
In the case of Type~Dp, the morphism $\widetilde{X}\to \mP^{1}$ is not, in general, contained in a $\mP(1^{2},2,3)$-bundle.
\end{rem}

The existence of the Sarkisov link in Theorem~\ref{thm:keySarkisov} is nontrivial, whereas its description is relatively straightforward. 
On the other hand, Theorem~\ref{thm:FanoSarkisov} is obtained by restricting the situation of Theorem~\ref{thm:keySarkisov}, so its construction is more or less straightforward, although its description, while computationally elementary, is somewhat involved.

\vspace{3pt}


\vspace{3pt}

\noindent \textbf{Acknowledgment.} 
This work was supported by JSPS KAKENHI Grant Number 25K06923.

\vspace{5pt}

\noindent \textbf{Notation.} 
In what follows, $*$ denotes any of the coordinates $p_{i}\,(1\leq i\leq 4)$, $u_{1}$, $u_{2}$, $t_{1}$, $t_{2}$, $t_{123}$, $t_{124}$, $t_{125}$, $t_{126}$, $t_{135}$, $t_{136}$, $t_{245}$ of $\Pi^{14}_{\mA}$. 
We denote by $z$ the coordinate of weight $1$ newly added for $\Pi_{\mP}^{14}$ in Type~Dp. 
The symbols $w(z)$, $w(s_{i})\,(i=1,2,3)$, and $w(*)$ denote the weights of the coordinates of $\Pi_{\mP}$.

\section{\textbf{Toric VGIT}}\label{sec:Toric-Sarkisov-link}

The contents of this section follow immediately from the general theory of toric VGIT described in \cite{BZ} and \cite{CoLS}, once we specify the settings given by \eqref{eq:wtmatdp}, \eqref{eq:wtmatDp1}, and \eqref{eq:wtmatdiv} below.

\subsection{First toric VGIT}\label{subsec:The-first-toric}

Let $\Pi_{\mA}$ denote the affine cone of $\Pi_{\mP}$. 
Let $\mA_{\mathtt{I}}$ be the affine space obtained by adding one more coordinate $w_{1}$ to the ambient affine space of $\Pi_{\mA}$. 
In this subsection, we consider the $(\mC^{*})^{2}$-action on $\mA_{\mathtt{I}}$ whose weights of coordinates are given by the following matrix:

\begin{equation}\label{eq:wtmatdp}
\setlength{\arraycolsep}{4pt}
\renewcommand{\arraystretch}{1.05}
\begin{array}{@{} c *{5}{c} c @{}}
  & w_1 & s_1 & s_2 & s_3 & * & \\[2pt]
  \multirow{2}{*}{$\Bigl(\!$}
   & 0  & w(s_1)  & w(s_2)  & w(s_3)  & w(*) 
   & \multirow{2}{*}{$\!\Bigr).$} \\
   & -1 & -1 & -1 & -1 & 0 &
\end{array}
\end{equation}

By \cite[p.~517, Table~3]{Tak6}, we note that in all cases
\begin{equation}
w(s_{2}) = w(s_{1}) + 1, \qquad w(s_{3}) = w(s_{1}) + 2
\label{eq:ws2}
\end{equation}
holds.

The GIT quotients corresponding to this matrix form the following diagram:
\begin{equation}\label{eq:first}
\xymatrix@C=15pt{
& \mathcal{T}_{-2}\ar[dl]\ar[dr]  && \mathcal{T}_{-1}\ar[dl]\ar[dr] && \mathcal{T}_0\ar[dl]\ar[dr] \\
\overline{\mathcal{T}}_{-2} && \overline{\mathcal{T}}_{-1} && \overline{\mathcal{T}}_0 && \overline{\mathcal{T}}_1.
}
\end{equation}
Here, $\overline{\mathcal{T}}_{-2}$, $\overline{\mathcal{T}}_{-1}$, $\overline{\mathcal{T}}_{0}$, and $\overline{\mathcal{T}}_{1}$ are the GIT quotients corresponding to the $2$nd, $3$rd, $4$th, and $5$th columns of the matrix, respectively (we regard $\left(\begin{array}{c} w(*)\\ 0 \end{array}\right)$ as one column). 
The variety $\overline{\mathcal{T}}_{-2}$ is the ambient weighted projective space of $\Pi_{\mP}$, and $\overline{\mathcal{T}}_{1}$ is the weighted projective space obtained from $\overline{\mathcal{T}}_{-2}$ by eliminating the coordinates $s_{1}$, $s_{2}$, and $s_{3}$. 
Moreover, $\mathcal{T}_{-2}$, $\mathcal{T}_{-1}$, and $\mathcal{T}_{0}$ are the GIT quotients corresponding to the chambers generated by the $2$nd and $3$rd columns, the $3$rd and $4$th columns, and the $4$th and $5$th columns, respectively, and they are $\mQ$-factorial toric varieties of Picard number $2$.

Taking \eqref{eq:ws2} into account, the morphism $\mathcal{T}_{-2}\to\overline{\mathcal{T}}_{-2}$ is given by
\begin{equation}
(w_{1}, s_{1}, s_{2}, s_{3}, *) \mapsto 
\bigl(s_{1},\, w_{1}^{\nicefrac{1}{w(s_{1})}} s_{2},\, w_{1}^{\nicefrac{2}{w(s_{1})}} s_{3},\, w_{1}^{\nicefrac{w(*)}{w(s_{1})}} *\bigr),
\label{eq:toricwtblup1st}
\end{equation}
that is, $\mathcal{T}_{-2}\to\overline{\mathcal{T}}_{-2}$ is the weighted blow-up of weight $\nicefrac{1}{w(s_{1})}(1,2,w(*))$.

\subsection{Second toric VGIT}\label{subsec:The-second-toric}

Let $\mA_{\mathtt{II}}$ be the affine space obtained by adding the coordinates $w_{2}$, $s_{123}$, $s_{124}$, $s_{125}$, $s_{126}$, $s_{135}$, $s_{136}$, $s_{245}$, and $s_{246}$ to the ambient affine space of the affine cone of $\overline{\mathcal{T}}_{1}$. 
In this subsection, we consider a $(\mC^{*})^{2}$-action on $\mA_{\mathtt{II}}$. 
The meaning of the affine space $\mA_{\mathtt{II}}$ and its $(\mC^{*})^{2}$-action will be clarified in Section~\ref{sec:Sarkisov-link-for key var}. 
We treat the cases of Type~Dp and Type~Div separately.

\subsubsection{\textbf{\textup{Type~Dp}}}

In this case, the $(\mC^{*})^{2}$-action is given by the following matrix:
\begin{equation}\label{eq:wtmatDp1}
\setlength{\arraycolsep}{4pt}
\renewcommand{\arraystretch}{1.05}
\begin{array}{@{} c *{8}{c} c @{}}
  &w_2& * & s_{246} & s_{245} & s_{126} & \substack{s_{136}\\s_{124}}
   & s_{125} & \substack{s_{135}\\s_{123}} & \\[2pt]
  \multirow{2}{*}{$\Bigl(\!$}
   & 0  & w(*)  & 3d+3  & 3d+1  & 3d  & 3d-1  & 3d-2  & 3d-3 & \multirow{2}{*}{$\!\Bigr).$} \\
   & -1 & 0 & 1  & 1  & 1  & 1  & 1  & 1 &
\end{array}
\end{equation}
The GIT quotients corresponding to this matrix form the following diagram:
\begin{equation}\label{eq:ToricSarkisovDel}
\xymatrix@C=8pt{
& \mathcal{T}_1\ar[dl]\ar[dr]  && \mathcal{T}_2\ar[dl]\ar[dr]
&& \mathcal{T}_3\ar[dl]\ar[dr] && \mathcal{T}_4\ar[dl]\ar[dr]
&& \mathcal{T}_5\ar[dl]\ar[dr] && \mathcal{T}_6\ar[dl]\ar[dr] &\\
\overline{\mathcal{T}}'_1 && \overline{\mathcal{T}}_2 && \overline{\mathcal{T}}_3
&& \overline{\mathcal{T}}_4 && \overline{\mathcal{T}}_5
&& \overline{\mathcal{T}}_6 && \mP^1.
}
\end{equation}
Here, $\overline{\mathcal{T}}'_{1}$ corresponds to the $2$nd column of the matrix, and $\mathcal{T}_{1}$ corresponds to the chamber generated by the $2$nd and $3$rd columns. 
As in Subsection~\ref{subsec:The-first-toric}, the remaining columns of the matrix correspond, from left to right, to the GIT quotients appearing in the diagram. The variety $\overline{\mathcal{T}}'_{1}$ is the weighted projective space obtained from $\overline{\mathcal{T}}_{1}$ by adjoining the coordinates $s_{123}$, $s_{124}$, $s_{125}$, $s_{126}$, $s_{135}$, $s_{136}$, $s_{245}$, and $s_{246}$ with weights given by the first row of the matrix \eqref{eq:wtmatDp1}. The rightmost $\mP^{1}$ in \eqref{eq:ToricSarkisovDel} corresponds to the last column of the matrix \eqref{eq:wtmatDp1}, and $(s_{123}:s_{135})$ can be regarded as its homogeneous coordinates.

\subsubsection{\textbf{\textup{Type~Div}}}

In this case, the $(\mC^{*})^{2}$-action is given by the following matrix:
\begin{equation}\label{eq:wtmatdiv}
\setlength{\arraycolsep}{4pt}
\renewcommand{\arraystretch}{1.05}
\begin{array}{@{} c *{8}{c} c @{}}
  & w_2 & * & s_{246} & s_{245}
  & \substack{s_{126}\\s_{136}}
  & \substack{s_{124}\\s_{125}\\s_{135}}
  & s_{123} & \\[2pt]
  \multirow{2}{*}{$\Bigl(\!$}
   & 0  & w(*)  & 3d+6  & 3d+5  & 3d+4  & 3d+3  & 3d+2
   & \multirow{2}{*}{$\!\Bigr).$} \\
   & -1 & 0 & 1  & 1  & 1  & 1  & 1 &
\end{array}
\end{equation}
The GIT quotients corresponding to this matrix form the following diagram:
\begin{equation}\label{eq:ToricSarkisovDiv}
\xymatrix@C=15pt{
& \mathcal{T}_1\ar[dl]\ar[dr]  && \mathcal{T}_2\ar[dl]\ar[dr]
&& \mathcal{T}_3\ar[dl]\ar[dr] && \mathcal{T}_4\ar[dl]\ar[dr]\\
\overline{\mathcal{T}}'_1 && \overline{\mathcal{T}}_2 && \overline{\mathcal{T}}_3
&& \overline{\mathcal{T}}_4 && \overline{\mathcal{T}}_5.
}
\end{equation}
Here, $\overline{\mathcal{T}}'_{1}$ corresponds to the $2$nd column of the matrix, and $\mathcal{T}_{1}$ corresponds to the chamber generated by the $2$nd and $3$rd columns. 
As in Subsection~\ref{subsec:The-first-toric}, the remaining columns of the matrix correspond, from left to right, to the GIT quotients appearing in the diagram. The space $\overline{\mathcal{T}}'_{1}$ is the weighted projective space obtained from $\overline{\mathcal{T}}_{1}$ by adjoining the coordinates $s_{123}$, $s_{124}$, $s_{125}$, $s_{126}$, $s_{135}$, $s_{136}$, $s_{245}$, and $s_{246}$ with weights given by the first row of the matrix \eqref{eq:wtmatdiv}.

\section{\textbf{Sarkisov link for the key variety $\Pi_{\mP}$--Proof of Theorem~\ref{thm:keySarkisov}--}}\label{sec:Sarkisov-link-for key var}

In this section, we construct the Sarkisov link for $\Pi_{\mP}$ induced by the toric VGIT studied in Subsections~\ref{subsec:The-first-toric} and \ref{subsec:The-second-toric}, and prove Theorem~\ref{thm:keySarkisov}.

\subsection{Restriction of the diagram \eqref{eq:first}}\label{subsec:The-restriction-of 1st}

Let $\Pi_{-2}$ be the strict transform of $\Pi_{\mP}$ under the morphism $\mathcal{T}_{-2}\to\overline{\mathcal{T}}_{-2}$. 
Its defining equations $F_{1}$--$F_{9}$ are obtained from the nine equations $G_{1}$--$G_{9}$ given in \cite[A.2]{Tak6} (also listed in \cite[Sec.~3]{Tak9}) by setting $t_{246}=1$, substituting \eqref{eq:toricwtblup1st}, and dividing out powers of $w_{1}$ as much as possible. 
Since these equations are quite lengthy, we write them explicitly in Subsection~\ref{subsec:The-equation-ofPi0}.

Let $\Pi_{-1}$ and $\Pi_{0}$ be the subvarieties of $\mathcal{T}_{-1}$ and $\mathcal{T}_{0}$, respectively, defined by the same equations as $\Pi_{-2}$. 
By restricting the diagram \eqref{eq:first} to $\Pi_{\mP}$, $\Pi_{-2}$, $\Pi_{-1}$, and $\Pi_{0}$, we obtain the following diagram:
\begin{equation*}
\xymatrix@C=15pt{
& \Pi_{-2}\ar[dl]\ar[dr]  && \Pi_{-1}\ar[dl]\ar[dr] && \Pi_0\ar[dl]\ar[dr] \\
\Pi_{\mP} && \overline{\Pi}_{-1} && \overline{\Pi}_0 && \overline{\Pi}_1,
}
\end{equation*}
where $\overline{\Pi}_{-1}$, $\overline{\Pi}_{0}$, and $\overline{\Pi}_{1}$ are defined as the images of the corresponding morphisms appearing in \eqref{eq:first}.

By analyzing the unstable loci of the GIT quotients and the exceptional loci via $F_1$--$F_9$, we see that all the morphisms are isomorphisms except for $\Pi_{-2}\to\Pi_{\mP}$ and $\Pi_{0}\to\overline{\Pi}_{1}$. 
Thus, we may simply rewrite the restriction of \eqref{eq:first} as
\begin{equation}\label{eq:firstres}
\xymatrix{
& \widehat{\Pi}\ar[dl]_{\widehat{f}_{\Pi}}\ar[dr]^{\widehat{g}_{\Pi}} \\
\Pi_\mP && \overline{\Pi}_1.
}
\end{equation}

From the polynomials $F_{1}$, $F_{2}$, $F_{7}$, $F_{8}$ in Subsection~\ref{subsec:The-equation-ofPi0} and the corresponding defining equations $G_{1}$, $G_{2}$, $G_{7}$, $G_{8}$ of $\Pi_{\mP}$, we can eliminate the coordinates $u_{1}$, $u_{2}$, $s_{3}$, and $t_{2}$ from the orbifold coordinates of $\widehat{\Pi}$ near the exceptional divisor $\{w_{1}=0\}$ of $\widehat{\Pi}\to\Pi_{\mP}$ and from the orbifold coordinates of $\Pi_{\mP}$ near the $s_{1}$-point. 
Therefore, by \eqref{eq:toricwtblup1st}, the morphism $\widehat{\Pi}\to\Pi_{\mP}$ can be described near the exceptional divisor as
\begin{equation}
(w_{1}, s_{1}, s_{2}, *') \mapsto \bigl(s_{1},\, w_{1}^{\nicefrac{1}{w(s_{1})}} s_{2},\, w_{1}^{\nicefrac{w(*')}{w(s_{1})}} *'\bigr),
\label{eq:Piwtblup}
\end{equation}
where $*'$ denotes the collection of coordinates obtained from $*$ by removing $u_{1}$, $u_{2}$, $s_{3}$, and $t_{2}$. 
In other words, $\widehat{\Pi}\to\Pi_{\mP}$ is the weighted blow-up of weight $\nicefrac{1}{w(s_{1})}(1,w(*'))$. 
Since $\Pi_{\mP}$ is $\mQ$-factorial and has Picard number $1$ (cf.~\cite[Prop.~5.10]{Tak5}), it follows that $\widehat{\Pi}$ is $\mQ$-factorial and has Picard number $2$. 

Moreover, the morphism $\widehat{\Pi}\to\overline{\Pi}_{1}$ is obtained by eliminating the coordinates $w_{1}$, $s_{1}$, $s_{2}$, $s_{3}$.

\subsection{Description of $\overline{\Pi}_{1}$}\label{subsec:Description-of Pi1}

Consider the following matrix:
\begin{equation}\label{eq:matbasic}
\setlength{\arraycolsep}{3pt}
\mathsf{M}:=\begin{pmatrix}
u_{1} & u_{2} & -t_{2}p_{2} & -t_{2}p_{4} & -t_{2}p_{1}+t_{1}u_{1} & -t_{2}p_{3}+t_{1}u_{2}\\
-p_{1} & -p_{3} & u_{1}+t_{1}p_{2} & u_{2}+t_{1}p_{4} & -t_{2}p_{2} & -t_{2}p_{4}\\
-p_{2} & -p_{4} & -p_{1} & -p_{3} & u_{1} & u_{2}
\end{pmatrix}.
\end{equation}
Let $D_{ijk}$ denote the determinant of the $3\times3$ submatrix obtained by taking the $i$, $j$, and $k$ columns ($i<j<k$) of this matrix in this order. 
Then it is easy to check that the ideal $(D_{ijk}\mid1\le i<j<k\le6)$ is generated by
\begin{equation}
D_{123},D_{124},D_{125},D_{126},D_{135},D_{136},D_{245},D_{246}.
\label{eq:Dijk}
\end{equation}
We define
\begin{align*}
E_{\overline{\Pi}} & :=\{D_{123}=D_{124}=D_{125}=D_{126}=D_{135}=D_{136}=D_{245}=D_{246}=0\}\subset\overline{\mathcal{T}}_{1},\\
E_{\widehat{\Pi}} & :=\widehat{\Pi}\cap\{w_{1}=0\}.
\end{align*}

\begin{lem}
\label{lem:The-image-of=003000excep}
The image of $\,E_{\widehat{\Pi}}$ under $\widehat{g}_{\Pi}\colon\widehat{\Pi}\to\overline{\Pi}_{1}$ coincides with $E_{\overline{\Pi}}$. 
In particular, $E_{\overline{\Pi}}$ is irreducible. 
Moreover, $E_{\widehat{\Pi}}\to E_{\overline{\Pi}}$ is a finite morphism.
\end{lem}

\begin{proof}
Note that $s_1\not=0$ near $E_{\widehat{\Pi}}$. Thus we may assume that $s_1=1$  near $E_{\widehat{\Pi}}$, and then 
the equations of $E_{\widehat{\Pi}}$ can be reduced to 
\begin{equation}
s_{3}=s_{2}^{2},\quad t_{2}=s_{2}^{3}+s_{2}t_{1},\quad u_{1}=p_{1}s_{2}+p_{2}s_{2}^{2},\quad u_{2}=p_{3}s_{2}+p_{4}s_{2}^{2}.
\label{eq:paramexcep}
\end{equation}
By omitting the relation $s_{3}=s_{2}^{2}$, this turns out to give a parametrization of $E_{\overline{\Pi}}$ by elimination theory, and we checks that this parametrization coincides with the restriction of $\widehat{g}_{\Pi}$ to $E_{\widehat{\Pi}}$. 
This proves the first assertion. 

For the second assertion, suppose that $E_{\widehat{\Pi}}\to E_{\overline{\Pi}}$ is not finite. 
Since $E_{\widehat{\Pi}}$ is a weighted projective space, this would imply that $E_{\overline{\Pi}}$ consists of a single point, which is clearly not the case.
\end{proof}

We introduce the following polynomial:
\begin{equation}\label{eq:Hyp}
\begin{aligned}
F_{\overline{\Pi}}
&:= \widetilde{t}_{123}D_{123}+\widetilde{t}_{124}D_{124}
   +t_{125}D_{125}+t_{126}D_{126}\\
&\quad +t_{135}D_{135}+\widetilde{t}_{136}D_{136}
   +\widetilde{t}_{245}D_{245}+D_{246},
\end{aligned}
\end{equation}
where we set
\[
\widetilde{t}_{123}:=t_{123}+2t_{1}t_{136},\quad
\widetilde{t}_{124}:=t_{124}-2t_{1}t_{245},\quad
\widetilde{t}_{136}:=3t_{136},\quad
\widetilde{t}_{245}:=3t_{245}.
\]

\begin{lem}
\label{lem:HypNormal}
The following hold:
\begin{enumerate}[$(1)$]
\item[(1)] The affine hypersurface $\{F_{\overline{\Pi}}=0\}$ is normal.
\item[(2)] The weighted projective hypersurface $\{F_{\overline{\Pi}}=0\}\subset\overline{\mathcal{T}}_{1}$ is normal. In particular, it is irreducible.
\end{enumerate}
\end{lem}

\begin{proof}
(1) It suffices to show that $\overline{\Pi}_{\mA}$ is regular in codimension $1$. 
By considering the partial derivatives of $F_{\overline{\Pi}}$ with respect to the variables $t_{ijk}$, we see that the singular locus of $\overline{\Pi}_{\mA}$ is contained in the affine cone of $E_{\overline{\Pi}}$. 
By Lemma~\ref{lem:The-image-of=003000excep}, the affine cone of $E_{\overline{\Pi}}$ is irreducible and has codimension $1$ in $\{F_{\overline{\Pi}}=0\}$, so it suffices to show that $\overline{\Pi}_{\mA}$ is smooth at some point of the affine cone of $E_{\overline{\Pi}}$. 
For instance, we may verify that $\overline{\Pi}_{\mA}$ is smooth at the point with
\[
p_{1}=1,\quad p_{2}=p_{3}=0,\quad p_{4}=1,\quad u_{1}=u_{2}=1,\quad t_{1}=0,\quad t_{2}=1,\quad t_{ijk}=0\ \text{for all } i,j,k.
\]

(2) By the weights of the coordinates of $\overline{\mathcal{T}}_{1}$, the $\mC^{*}$-action on the affine hypersurface $\{F_{\overline{\Pi}}=0\}$ is free in codimension $1$. 
Hence the first assertion implies the first part of (2).
Since $\{F_{\overline{\Pi}}=0\}$ is an ample divisor in $\overline{\mathcal{T}}_{1}$, it is connected, and therefore irreducible.
\end{proof}

\begin{prop}
\label{prop:widehat=00007BPi=00007D bir}
It holds that $\,\overline{\Pi}_{1}=\{F_{\overline{\Pi}}=0\}$. 
The morphism $\widehat{g}_{\Pi}\colon\widehat{\Pi}\to\overline{\Pi}_{1}$ is a crepant small contraction, and every nontrivial fiber is $1$-dimensional.
\end{prop}

\begin{proof}
For each point of $\overline{\mathcal{T}}_{1}$, the fiber of $\widehat{\Pi}\to\overline{\mathcal{T}}_{1}$ is the subset defined by $G_{1}$--$G_{9}$ in $\mP^{3}$ with the coordinates $s_{1},s_{2},s_{3},w_{1}$, where all other variables of $G_{1}$--$G_{9}$ are regarded as constants. Note that $G_{1}$--$G_{6}$ are linear and $G_7$--$G_9$ are quadratic in $s_{1},s_{2},s_{3},w_{1}$. 

We check that the ideal generated by all the $4\times4$ minors of the $4\times6$ matrix $\mathsf{L}$ determined by the coefficients of $s_{1},s_{2},s_{3},w_{1}$ in $G_{1}$--$G_{6}$ is generated by the products of the $2\times2$ minors of $\mathsf{M}$ and $F_{\overline{\Pi}}$ (\cite[Sec.~3]{Tak9}). Since $F_{\overline{\Pi}}$ is a linear combination of the $3\times3$ minors of $\mathsf{M}$, it vanishes at every point where all $2\times2$ minors of $\mathsf{M}$ vanish. 
Thus, at a point where $F_{\overline{\Pi}}\ne0$, we have $\rank\mathsf{L}=4$, and hence the equations $G_{1}=\cdots=G_{6}=0$ imply $s_{1}=s_{2}=s_{3}=w_{1}=0$. Therefore the fiber of $\widehat{\Pi}\to\overline{\mathcal{T}}_{1}$ over such a point is empty, and hence $\overline{\Pi}_{1}\subset\{F_{\overline{\Pi}}=0\}$.

By Lemma~\ref{lem:HypNormal} (2), the hypersurface $\{F_{\overline{\Pi}}=0\}$ is irreducible. 
Moreover, it has the same dimension as $\widehat{\Pi}$. 
In view of these facts, we first show that $\widehat{\Pi}\to\overline{\Pi}_{1}$ is birational. 
It suffices to show that the fiber over any point of $\{F_{\overline{\Pi}}=0\}\setminus E_{\overline{\Pi}}$ consists of a single point.
For any $D_{ijk}$, we find that among the equations $G_{1}$--$G_{6}$, there exist three independent linear equations in $s_{1},s_{2},s_{3}$ which can be uniquely solved at a point where $D_{ijk}\ne0$ (for instance, for $D_{123}$ we can take $G_{1}=G_{2}=G_{3}=0$). 
We checks that this solution also satisfies $G_{1}$--$G_{9}$, and hence the fiber consists of a single point.

Next, we show that every nontrivial $\widehat{g}_{\Pi}$-fiber  is $1$-dimensional. 
By Lemma \ref{lem:The-image-of=003000excep}, the restriction of $\widehat{g}_{\Pi}$ to $\widehat{\Pi}\cap\{w_{1}=0\}$ is finite. 
From the above argument, the image of every nontrivial $\widehat{g}_{\Pi}$-fiber is contained in $E_{\overline{\Pi}}$. 
Thus every nontrivial $\widehat{g}_{\Pi}$-fiber intersects $\{w_{1}=0\}$ in a finite nonempty set, and hence must be $1$-dimensional.

The crepancy of $\widehat{g}_{\Pi}$ follows from the computation of the canonical divisor of $\Pi_{\mP}$ (cf.~\cite[Prop.~5.5]{Tak5}) and the choice of the weighted blow-up $\widehat{f}_{\Pi}$.

It remains to show that $\widehat{g}_{\Pi}$ is small. 
Since it is crepant, the image of the exceptional locus is contained in the singular locus of $\overline{\Pi}_{1}$. 
As $\widehat{\Pi}$ has Picard number $2$, the relative Picard number of  $\widehat{g}_{\Pi}$ is $1$. 
Hence, if $\widehat{g}_{\Pi}$ were divisorial, then its exceptional locus would be a prime divisor. 
Since every nontrivial fiber is $1$-dimensional, this would imply that the singular locus of $\overline{\Pi}_{1}$ has an irreducible component of codimension $2$. We can check that there exists a nontrivial $\widehat{g}_{\Pi}$-fiber with $p_{4}\ne0$ (for example, the fiber over the $p_{4}$-point). 
Hence the image of the exceptional divisor would not be contained in $\{p_{4}=0\}$, and thus $\overline{\Pi}_{1}\cap\{p_{4}\ne0\}$ would have a singular locus of codimension $2$. 
It therefore suffices to show that the singular locus of $\overline{\Pi}_{1}\cap\{p_{4}\ne0\}$ has codimension $3$. 

Note that the singularities arising from the quotient of the smooth locus of $\{F_{\overline{\Pi}}|_{p_{4}=1}=0\}$ have codimension at least $3$ by the weights of the coordinates. 
Thus it suffices to show that the singular locus of $\{F_{\overline{\Pi}}|_{p_{4}=1}=0\}$ has codimension $3$. 
We omit the detailed computation and refer to \cite[Sec.~4]{Tak9}; we only sketch the argument. 
Along $p_{4}=1$, the common zero locus of the $2\times2$ minors of $\mathsf{M}$ is defined by four polynomials, and we can construct a polynomial coordinate change in which these four polynomials form part of a coordinate system. 
In this coordinate system, the singular locus of $\{F_{\overline{\Pi}}|_{p_{4}=1}=0\}$ can be determined explicitly, and we see that it has codimension $3$.
\end{proof}

\begin{rem}
The birational map $\overline{\Pi}_{1}\dashrightarrow\Pi_{\mP}$ can be regarded as a generalization of the Type ${\rm II}_{2}$ unprojection studied in \cite[Sec.~5.2]{Tay}. 
The title of this paper is motivated by this observation.
\end{rem}

Up to this point, the results of this paper, when restricted to $3$-folds, may not be entirely new to experts, apart from the reinterpretation in the framework of VGIT and a careful treatment of the technical details.
However, how to develop birational geometry further from the diagram \eqref{eq:firstres} seems to have been unknown even in the $3$-fold case. 
We think that providing this development is one of the main contributions of this paper. 
In the following subsections, we construct the key ingredients for this purpose. 
For these constructions, we use the diagrams \eqref{eq:ToricSarkisovDel} and \eqref{eq:ToricSarkisovDiv} in the cases of Type~Dp and Type~Div, respectively.

\subsection{Embedding $\overline{\Pi}_{1}$ into $\overline{\mathcal{T}}_{1}'$}

Set
\begin{equation}
I:=\widetilde{t}_{123}s_{123}+\widetilde{t}_{124}s_{124}+t_{125}s_{125}+t_{126}s_{126}+t_{135}s_{135}+\widetilde{t}_{136}s_{136}+\widetilde{t}_{245}s_{245}+s_{246},
\label{eq:hypst}
\end{equation}
which is obtained from $F_{\overline{\Pi}}$ by replacing each $D_{ijk}$ with $s_{ijk}$. 
Here $(i,j,k)$ runs over $(1,2,3)$, $(1,2,4)$, $(1,2,5)$, $(1,2,6)$, $(1,3,5)$, $(1,3,6)$, $(2,4,5)$, and $(2,4,6)$. In what follows, when we write $s_{ijk}$ or $D_{ijk}$, we always assume that $(i,j,k)$ ranges over this set. 
Using this notation, we identify the hypersurface $\overline{\Pi}_{1}=\{F_{\overline{\Pi}}=0\}\subset\overline{\mathcal{T}}_{1}$ with
\begin{equation}
	\label{eq:Pi1new}
\{I=0,s_{ijk}=D_{ijk}\}\subset\overline{\mathcal{T}}_{1}'.
\end{equation}
In this new $\overline{\Pi}_{1}$, we have
\[
E_{\overline{\Pi}}=\overline{\Pi}_{1}\cap\{s_{ijk}=0\}.
\]
\subsection{Constructing the strict transforms of $\overline{\Pi}_{1}$ for the
diagrams \eqref{eq:ToricSarkisovDel} and \eqref{eq:ToricSarkisovDiv}}
\label{subsec:Constructing-strict-transforms Pi1}

Recall that the affine space $\mA_{\mathtt{II}}$ was introduced at the beginning of Subsection~\ref{subsec:The-second-toric}. 
Note that the morphism $\mathcal{T}_{1}\to\overline{\mathcal{T}}_{1}'$ is given by
\begin{equation}
(w_{2}, *, s_{ijk}) \mapsto (*, w_{2}s_{ijk}).
\label{eq:T1T1'}
\end{equation}

In the GIT quotient $\mathcal{T}_{1}$ of $\mA_{\mathtt{II}}$, we set 
\begin{equation}
\Pi_{1}':=\{I=0,\; w_{2}s_{ijk}=D_{ijk}\},
\label{eq:Pi1'}
\end{equation}
where $(i,j,k)$ ranges over the set specified above. By \eqref{eq:Pi1new} and \eqref{eq:Pi1'}, the image of $\Pi_{1}'$ by the morphism $\mathcal{T}_{1}\to\overline{\mathcal{T}}_{1}$ is contained in $\overline{\Pi}_{1}$. Setting $w_{2}=1$ in \eqref{eq:Pi1'}, we see that $\Pi_{1}'$ dominates $\overline{\Pi}_{1}$ via the morphism $\mathcal{T}_{1}\to\overline{\mathcal{T}}_{1}$. Moreover, again by \eqref{eq:Pi1new} and \eqref{eq:Pi1'}, we see that the inverse image by $\Pi_{1}'\to \overline{\Pi}_{1}$ of a point in $\overline{\Pi}_{1}\setminus E_{\overline{\Pi}}$ consists of one point.  
However, $\Pi_{1}'$ is not irreducible, since it contains a component given by $\{I=w_2=D_{ijk}=0\}$, which dominates $E_{\overline{\Pi}}$.

To extract the irreducible component of $\Pi_{1}'$ that birationally dominates $\overline{\Pi}_{1}$, we first consider, inside the open subset $\{s_{135}\neq 0\}$ of $\mA_{\mathtt{II}}$, the subvariety defined by $I=0$ and the following equations:
\begin{equation}
\begin{aligned}
w_{2} & =\tfrac{1}{s_{135}}\,D_{135},\qquad 
p_{3}=-\tfrac{1}{s_{135}}\bigl(-p_{1}(s_{136}+s_{123}t_{1})+p_{2}(s_{125}t_{1}+s_{123}t_{2})+s_{125}u_{1}\bigr),\\[2pt]
p_{4} & =\tfrac{1}{s_{135}}\bigl(p_{1}s_{125}+p_{2}(s_{136}+s_{123}t_{1})+s_{123}u_{1}\bigr),\quad 
u_{2}=\tfrac{1}{s_{135}}\bigl(p_{1}s_{123}t_{2}-p_{2}s_{125}t_{2}+s_{136}u_{1}\bigr),\\[2pt]
s_{124} & =\tfrac{1}{s_{135}}\bigl(s_{125}^{2}+s_{123}s_{136}+s_{123}^{2}t_{1}\bigr),\quad 
s_{126}=\tfrac{1}{s_{135}}\bigl(s_{125}s_{136}-s_{123}^{2}t_{2}\bigr),\\[2pt]
s_{245} & =\tfrac{1}{s_{135}}\bigl(s_{136}^{2}+s_{125}^{2}t_{1}+2s_{123}s_{136}t_{1}+s_{123}^{2}t_{1}^{2}+s_{123}s_{125}t_{2}\bigr),\\[2pt]
s_{246} & =\tfrac{1}{s_{135}^{2}}\bigl(s_{136}^{3}+s_{125}^{2}s_{136}t_{1}+2s_{123}s_{136}^{2}t_{1}+s_{123}^{2}s_{136}t_{1}^{2}+s_{125}^{3}t_{2}\\
& \hspace{4.8em}{}+3s_{123}s_{125}s_{136}t_{2}+s_{123}^{2}s_{125}t_{1}t_{2}-s_{123}^{3}t_{2}^{2}\bigr).
\end{aligned}
\label{eq:s135}
\end{equation}
In the first equation of \eqref{eq:s135}, we use $D_{135}$ for notational simplicity; however, if we substitute the expressions for $p_{3}$, $p_{4}$, and $u_{2}$ given in the subsequent equations into $D_{135}$, then \eqref{eq:s135} can be regarded as eliminating the coordinates $w_{2}$, $p_{3}$, $p_{4}$, $u_{2}$, $s_{124}$, $s_{245}$, and $s_{246}$ from the coordinates of $\mA_{\mathtt{II}}$ (we will use the same idea immediately after Lemma~\ref{lem:sijkNull}). 
Moreover, by $I=0$, we can eliminate $t_{135}$. 
Thus, this subvariety itself is an open subset of an affine space.

Let $\Pi_{1}^{o}$ denote the subvariety of $\mathcal{T}_{1}\cap\{s_{135}\not=0\}$ defined by \eqref{eq:s135} and $I=0$. 
Then $\Pi_{1}^{o}$ has the same dimension as $\Pi_{\mP}$ and has at worst quotient singularities. 
We check that $\Pi_{1}^{o}\subset\Pi_{1}'$ and $\Pi_{1}^{o}$ dominates $\overline{\Pi}_{1}$. 

Let $\Pi_{\mathtt{II}}$ be the closure in $\mA_{\mathtt{II}}$ of the subvariety defined by $I=0$ and \eqref{eq:s135} inside $\{s_{135}\neq 0\}$. 
Although its defining equations can be obtained by elimination theory, they are extremely long, so we do not write them down here; see \cite[Sec.~5, closureEqs]{Tak9}. 
In what follows, we call an open subset of $\Pi_{\mathtt{II}}$ of the form $\{s_{ijk}\neq 0\}$ an \textit{$s_{ijk}$-chart} of $\Pi_{\mathtt{II}}$. 
We will also use the term “chart” for $\mA_{\mathtt{II}}$ and its GIT quotients.

\begin{lem}
\label{lem:sijkNull}
Any $s_{ijk}$-chart of $\,\Pi_{\mathtt{II}}$ is contained in one of the $s_{124}$-, $s_{135}$-, $s_{123}$-, or $s_{246}$-charts of $\,\Pi_{\mathtt{II}}$.
\end{lem}

\begin{proof}
By examining the defining equations of $\Pi_{\mathtt{II}}$, we see that if $s_{123}=s_{124}=s_{135}=s_{246}=0$, then all other $s_{ijk}$ also vanish. 
This proves the claim.
\end{proof}

The defining equations of the $s_{124}$-, $s_{135}$-, $s_{123}$-, and $s_{246}$-charts of $\Pi_{\mathtt{II}}$, excluding $I=0$, are as follows (see also \cite[Sec.~5]{Tak9}):

\noindent \textbf{The $s_{124}$-chart:} 
\begin{equation}
\begin{aligned}w_{2} & =\tfrac{1}{s_{124}}\,D_{124},\qquad p_{1}=\tfrac{1}{s_{124}}\bigl(p_{3}s_{123}+p_{4}s_{125}-p_{2}s_{126}\bigr),\\[2pt]
t_{1} & =\tfrac{1}{s_{124}^{2}}\bigl(-s_{126}^{2}+s_{124}s_{245}-s_{123}s_{246}\bigr),\qquad t_{2}=\tfrac{1}{s_{124}^{2}}\bigl(-s_{126}s_{245}+s_{125}s_{246}\bigr),\\[2pt]
u_{1} & =\tfrac{1}{s_{124}^{2}}\bigl(-p_{3}s_{124}s_{125}+p_{4}s_{125}s_{126}+p_{4}s_{123}s_{245}-p_{2}s_{124}s_{245}\bigr),\\[2pt]
u_{2} & =\tfrac{1}{s_{124}^{2}}\bigl(-p_{3}s_{124}s_{126}+p_{4}s_{126}^{2}+p_{4}s_{123}s_{246}-p_{2}s_{124}s_{246}\bigr),\\[2pt]
s_{135} & =\tfrac{1}{s_{124}^{2}}\bigl(s_{124}s_{125}^{2}+s_{123}s_{125}s_{126}+s_{123}^{2}s_{245}\bigr),\\
s_{136} & =\tfrac{1}{s_{124}^{2}}\bigl(s_{124}s_{125}s_{126}+s_{123}s_{126}^{2}+s_{123}^{2}s_{246}\bigr).
\end{aligned}
\label{eq:s124}
\end{equation}

\vspace{3pt}

\noindent \textbf{The $s_{135}$-chart:} we already write down this as (\ref{eq:s135}).

\vspace{3pt}

\noindent \textbf{The $s_{123}$-chart:} 
\begin{equation}
\begin{aligned}w_{2} & =\tfrac{1}{s_{123}}\,D_{123},\qquad p_{3}=\tfrac{1}{s_{123}}\bigl(p_{1}s_{124}+p_{2}s_{126}-p_{4}s_{125}\bigr),\\[2pt]
t_{1} & =-\tfrac{1}{s_{123}^{2}}\bigl(s_{125}^{2}-s_{124}s_{135}+s_{123}s_{136}\bigr),\qquad t_{2}=\tfrac{1}{s_{123}^{2}}\bigl(-s_{126}s_{135}+s_{125}s_{136}\bigr),\\[2pt]
u_{1} & =\tfrac{1}{s_{123}^{2}}\bigl(-p_{1}s_{123}s_{125}+p_{2}\bigl(s_{125}^{2}-s_{124}s_{135}\bigr)+p_{4}s_{123}s_{135}\bigr),\\[2pt]
u_{2} & =-\tfrac{1}{s_{123}^{2}}\bigl(p_{1}s_{123}s_{126}+p_{2}\bigl(s_{124}s_{136}-s_{125}s_{126}\bigr)-p_{4}s_{123}s_{136}\bigr),\\[2pt]
s_{245} & =\tfrac{1}{s_{123}^{2}}\bigl(-s_{124}s_{125}^{2}-s_{123}s_{125}s_{126}+s_{124}^{2}s_{135}\bigr),\\[2pt]
s_{246} & =\tfrac{1}{s_{123}^{2}}\bigl(-s_{124}s_{125}s_{126}-s_{123}s_{126}^{2}+s_{124}^{2}s_{136}\bigr).
\end{aligned}
\label{eq:s123}
\end{equation}

\vspace{3pt}

\noindent \textbf{The $s_{246}$-chart:} 
\begin{equation}
\begin{aligned}w_{2} & =\tfrac{1}{s_{246}}\,D_{246},\qquad p_{1}=\tfrac{1}{s_{246}}\bigl(p_{3}s_{245}-p_{3}s_{124}t_{1}+p_{4}s_{126}t_{1}+p_{4}s_{124}t_{2}+s_{126}u_{2}\bigr),\\[2pt]
p_{2} & =\tfrac{1}{s_{246}}\bigl(-p_{3}s_{126}+p_{4}s_{245}-p_{4}s_{124}t_{1}-s_{124}u_{2}\bigr),\\[2pt]
u_{1} & =\tfrac{1}{s_{246}}\bigl(-p_{3}s_{124}t_{2}+p_{4}s_{126}t_{2}+s_{245}u_{2}\bigr),\\[2pt]
s_{123} & =\tfrac{1}{s_{246}}\bigl(-s_{126}^{2}+s_{124}s_{245}-s_{124}^{2}t_{1}\bigr),\qquad s_{125}=\tfrac{1}{s_{246}}\bigl(s_{126}s_{245}+s_{124}^{2}t_{2}\bigr),\\[2pt]
s_{135} & =\tfrac{1}{s_{246}^{2}}\Bigl(s_{245}^{3}+s_{126}^{2}s_{245}t_{1}-2s_{124}s_{245}^{2}t_{1}+s_{124}^{2}s_{245}t_{1}^{2}-s_{126}^{3}t_{2}+3s_{124}s_{126}s_{245}t_{2}\\
 & \hspace{7.3em}{}-s_{124}^{2}s_{126}t_{1}t_{2}+s_{124}^{3}t_{2}^{2}\Bigr),\\[2pt]
s_{136} & =\tfrac{1}{s_{246}}\bigl(s_{245}^{2}+s_{126}^{2}t_{1}-2s_{124}s_{245}t_{1}+s_{124}^{2}t_{1}^{2}+s_{124}s_{126}t_{2}\bigr).
\end{aligned}
\label{eq:s246}
\end{equation}

\begin{defn}
By considering the unstable loci of the GIT quotients corresponding to $\mathcal{T}_{i}\,(i\geq1)$ in the diagrams \eqref{eq:ToricSarkisovDel} and \eqref{eq:ToricSarkisovDiv}, we see that at least one of the $s_{ijk}$ is nonzero on $\mathcal{T}_{i}$. 
Therefore, by Lemma~\ref{lem:sijkNull}, the image of the union of the $s_{124}$-, $s_{135}$-, $s_{123}$-, and $s_{246}$-charts of $\Pi_{\mathtt{II}}$ defines a projective variety in $\mathcal{T}_{i}$, which we
denote by $\Pi_{i}$.
\end{defn}

\begin{lem}
\label{lem:=00306F=009AD8=003005=005546=007279=007570=0070B9=003057=00304B=006301=00305F=00306A=003044=003002}
The variety $\,\Pi_{i}\, (i\geq 1)$ has the same dimension as $\Pi_{\mP}$ and has at worst quotient singularities.
\end{lem}

\begin{proof}
Note that each of the $s_{124}$-, $s_{135}$-, $s_{123}$-, and $s_{246}$-charts of $\Pi_{\mathtt{II}}$ is an open subset of an affine space. 
Moreover, we check that the $(\mC^{*})^{2}$-action is free in codimension $1$. 
The claim follows.
\end{proof}

\subsection{Restrictions of the diagrams \eqref{eq:ToricSarkisovDel} and
\eqref{eq:ToricSarkisovDiv}}\label{subsec:The-restrictions-of 2nd}

\subsubsection{\textbf{\textup{Type~Dp}}}

By restricting the diagram \eqref{eq:ToricSarkisovDel} to $\Pi_{i}\,(i\geq1)$, we obtain the following diagram:
\begin{equation}
	\label{eq:DpPi}
\xymatrix@C=8pt{
& \textcolor{gray}{\Pi_1}\ar@[gray][dl]\ar[dr]  && \Pi_2\ar[dl]\ar[dr]
&& \Pi_3\ar[dl]\ar[dr] && \textcolor{gray}{\Pi_4}\ar[dl]\ar@[gray][dr]
&& \textcolor{gray}{\Pi_5}\ar@[gray][dl]\ar[dr] && \textcolor{gray}{\Pi_6}\ar[dl]\ar@[gray][dr] &\\
\textcolor{gray}{\overline{\Pi}_1} && \overline{\Pi}_2 && \overline{\Pi}_3
&& \overline{\Pi}_4 && \textcolor{gray}{\overline{\Pi}_5}
&& \overline{\Pi}_6 && \textcolor{gray}{\mP^1},
}
\end{equation}
where $\overline{\Pi}_i$ are defined as the images of the corresponding morphisms appearing in \eqref{eq:ToricSarkisovDel}.
By examining the unstable loci of the GIT quotients and the exceptional loci via the equations of  $\Pi_{\mathtt{II}}$, we see that all the morphisms are isomorphisms except for 
$\Pi_{1}\to\overline{\Pi}_{1}$,
$\Pi_{4}\to\overline{\Pi}_{5}$,
$\Pi_{5}\to\overline{\Pi}_{5}$, and
$\Pi_{6}\to\mP^{1}$, which are highlighted in gray in the diagram \eqref{eq:DpPi}.
Thus, we may rewrite \eqref{eq:DpPi} as
\begin{equation}\label{eq:SarkisovDelkey}
\xymatrix@C=15pt{
& \Pi_1\ar[dl]_{\widetilde{g}_{\Pi}}\ar[dr]
&& \widetilde{\Pi}\ar[dl]\ar[dr]^{\widetilde{f}_\Pi} \\
\overline{\Pi}_1 && \mathring{\Pi} && \mP^1.
}
\end{equation}
We now study this diagram starting from the right-hand side.

\begin{prop}
\label{prop:transition-fcn}
The morphism $\,\widetilde{f}_{\Pi}\colon \widetilde{\Pi}\to \mP^{1}$ is a locally trivial 
$\,\mP(1^{2},2^{3},3^{2},4^{2},5,6,d-3,d-2,d-1)$-bundle over $\,\mP^{1}$ with homogeneous coordinates $\,(s_{135}:s_{123})$. 
In particular, $\widetilde{\Pi}$ has Picard number $\,2$. 

The fiber coordinates and the transition functions are given as follows, where
we omit coordinate transformations that differ only by a power of $\,\left(\frac{s_{135}}{s_{123}}\right)^{\pm1}$.

\noindent \textbf{$s_{135}\neq0:$}

Fiber coordinates:
\[
z,p_{1},p_{2},u_{1},t_{1},t_{2},t_{123},t_{124},t_{125},t_{126},t_{136},t_{245},
\frac{s_{125}}{s_{135}},\frac{s_{136}}{s_{135}}.
\]

Transition functions:
\begin{equation}
	\begin{aligned}
		t_{1} 
		&= -\left(\frac{s_{125}}{s_{123}}\right)^{2}
		+\frac{s_{135}}{s_{123}} \frac{s_{124}}{s_{123}}
		- \frac{s_{136}}{s_{123}},\\
		t_{2} 
		&= -\frac{s_{135}}{s_{123}}\frac{s_{126}}{s_{123}}
		+ \frac{s_{125}}{s_{123}}\frac{s_{136}}{s_{123}},\\
		u_{1} 
		&= -p_{1}\left(\frac{s_{125}}{s_{123}}\right)
		+ p_{2}\left(\left(\frac{s_{125}}{s_{123}}\right)^{2}
		- \frac{s_{135}}{s_{123}}\frac{s_{124}}{s_{123}}\right)
		+ \frac{s_{135}}{s_{123}}p_{4}.
	\end{aligned}
	\label{eq:s123-1}
\end{equation}

\vspace{3pt}

\noindent \textbf{$s_{123}\neq0:$}

Fiber coordinates:
\[
z,p_{1},p_{2},p_{4},t_{123},t_{124},t_{125},t_{126},t_{136},t_{245},
\frac{s_{124}}{s_{123}},\frac{s_{125}}{s_{123}},
\frac{s_{126}}{s_{123}},\frac{s_{136}}{s_{123}}.
\]
\begin{equation}
	\begin{aligned}
		\frac{s_{124}}{s_{123}} 
		&= \left(\frac{s_{123}}{s_{135}}\right)^{-1}\left(\frac{s_{125}}{s_{135}}\right)^{2}
		+ \frac{s_{136}}{s_{135}}
		+ \left(\frac{s_{123}}{s_{135}}\right)t_{1},\\
		\frac{s_{126}}{s_{123}} 
		&= \frac{s_{136}}{s_{135}} 
		- \left(\frac{s_{123}}{s_{135}}\right)t_{2},\\
		p_{4} 
		&=p_{1}\left(\frac{s_{125}}{s_{135}}\right) 
		+ p_{2}\left(\frac{s_{136}}{s_{135}} +\frac{s_{123}}{s_{135}} t_{1}\right)
		+ \left(\frac{s_{123}}{s_{135}}\right)u_{1}.
	\end{aligned}
	\label{eq:s135-1}
\end{equation}
\end{prop}

\begin{proof}
Note that the morphism $\widetilde{\Pi}\to\mP^{1}$ coincides with $\Pi_{6}\to\mP^{1}$. 
By examining the unstable locus of the GIT quotient defining $\Pi_{6}$, we see that at least one of $s_{135}$ or $s_{123}$ is nonzero on $\Pi_{6}$. 
Therefore, we may study $\widetilde{f}_{\Pi}$ using \eqref{eq:s135} and \eqref{eq:s123}.

On the $s_{123}$-chart, a fiber of $\widetilde{f}_{\Pi}$ is obtained as the weighted projective space given by removing from the coordinates of $\Pi_{\mathtt{II}}$ those appearing on the left-hand side of \eqref{eq:s123}, together with $t_{123}$, where $t_{123}$ can be eliminated using \eqref{eq:hypst}. The weights of these coordinates are read off from the first row of the matrix obtained by subtracting $(3d-3)$ times the second row from the first row of \eqref{eq:wtmatDp1}. 

The same argument applies to the $s_{135}$-chart. 
The transition functions are read directly from \eqref{eq:s135} and \eqref{eq:s123}.
\end{proof}

\begin{prop}
\label{prop:flipping}
The morphism $\,\Pi_{1}\to\mathring{\Pi}$ is a flipping contraction, and $\,\Pi_{1}\dashrightarrow\widetilde{\Pi}$ is its flip. 
The flip is a toric flip of type $(1^{2},2^{2},3,4^{2},5,6^{2},d-3,d-1,d,\ -2,-1)$. 

The flipping locus of $\ \Pi_{1}\to\mathring{\Pi}$ has coordinates 
$z$, $p_{2}$, $p_{3}$, $p_{4}$, $t_{123}$, $t_{125}$, $t_{126}$, $t_{135}$, $t_{136}$, $t_{245}$, $s_{126}$, $s_{245}$, $s_{246}$, 
and the flipped locus of  $\ \widetilde{\Pi}\to\mathring{\Pi}$ has coordinates $s_{125},s_{123}$.
\end{prop}

\begin{proof}
The claims follow from the defining equations of $\Pi_{\mathtt{II}}$ together with the description of exceptional loci in toric VGIT (see \cite[Lem.~4.5]{BZ} and \cite[Lem.~15.3.11]{CoLS}).

We recall that the morphism $\Pi_{1}\to\mathring{\Pi}$ coincides with $\Pi_{4}\to\overline{\Pi}_{5}$. 
Its exceptional locus is $\{s_{123}=s_{125}=s_{135}=0\}\cap\Pi_{4}$. 
Using the defining equations of $\Pi_{\mathtt{II}}$, we see that $s_{136}=0$ also holds on this locus. 
On the other hand, the unstable locus of the GIT quotient defining $\mathcal{T}_{4}$ contains $\{s_{123}=s_{124}=s_{125}=s_{135}=s_{136}=0\}$, so $s_{124}\neq0$ on the exceptional locus. 
Thus, we may work in the $s_{124}$-chart and set $s_{124}=1$. 

From the equations \eqref{eq:s124} of the $s_{124}$-chart, the coordinates 
$w_{2}, p_{1}, t_{1}, t_{2}, u_{1}, u_{2}, s_{135}, s_{136}$ can be eliminated, and $t_{124}$ can be eliminated using \eqref{eq:hypst}. 
The remaining coordinates are those listed in the statement, and hence the exceptional locus is a weighted projective space. 
Its weights are read from the matrix obtained by subtracting $(3d-1)$ times the second row from the first row of \eqref{eq:wtmatDp1}. 

The description of $\widetilde{\Pi}\to\mathring{\Pi}$ is obtained similarly. 
Both morphisms are small contractions, and by Proposition~\ref{prop:transition-fcn}, both $\widetilde{\Pi}$ and $\Pi_{1}$ have Picard number $2$. 
Their toric nature near the exceptional loci follows from the $s_{124}$-chart description and the weights of the coordinates. Therefore, $\Pi_{1}\dashrightarrow\widetilde{\Pi}$ is the flip of $\Pi_{1}\to\mathring{\Pi}$ with the type as stated.
\end{proof}

\subsubsection{\textbf{\textup{Type~Div}}}

By restricting the diagram \eqref{eq:ToricSarkisovDiv} to $\Pi_{i}\,(i\geq1)$, we obtain the following diagram:
\begin{equation}\label{eq:SarkisovDiv}
\xymatrix@C=15pt{
& \textcolor{gray}{\Pi_1}\ar@[gray][dl]\ar[dr]  && \Pi_2\ar[dl]\ar[dr]
&& \Pi_3\ar[dl]\ar[dr] && \textcolor{gray}{\Pi_4}\ar[dl]\ar@[gray][dr]\\
\textcolor{gray}{\overline{\Pi}_1} && \overline{\Pi}_2 && \overline{\Pi}_3
&& \overline{\Pi}_4 && \textcolor{gray}{\overline{\Pi}_5}.
}
\end{equation}

By examining the unstable loci of the GIT quotients and the exceptional loci via the equations of  $\Pi_{\mathtt{II}}$, we see that all the morphisms are isomorphisms except for $\Pi_{1}\to\overline{\Pi}_{1}$ and $\Pi_{4}\to\overline{\Pi}_{5}$, which are highlighted in gray in the diagram \eqref{eq:SarkisovDiv}. 
Thus, we may rewrite \eqref{eq:SarkisovDiv} as
\begin{equation}\label{eq:SarkisovDivkey}
\xymatrix{
& \widetilde{\Pi}\ar[dl]_{\widetilde{g}_\Pi}\ar[dr]^{\widetilde{f}_\Pi}\\
\overline{\Pi}_1 && \mathrm{w}\mP.
}
\end{equation}

\begin{prop}
\label{prop:Pi5WPS}
The variety $\mathrm{w}\mP$ is isomorphic to $\mP(1^{4},2^{4},3^{3},d-1,d^{2})$.
\end{prop}

\begin{proof}
We recall that the morphism $\widetilde{\Pi}\to w\mP$ coincides with $\Pi_{4}\to\overline{\Pi}_{5}$. 
Subtracting $(3d+2)$ times the second row from the first row of the matrix \eqref{eq:wtmatdiv}, we obtain the following matrix:
\begin{equation}\label{eq:wtmatdiv2}
\setlength{\arraycolsep}{4pt}
\renewcommand{\arraystretch}{1.05}
\begin{array}{@{} c *{8}{c} c @{}}
  & w_2 & * & s_{246} & s_{245}
  & \substack{s_{126}\\s_{136}}
  & \substack{s_{124}\\s_{125}\\s_{135}}
  & s_{123} & \\[2pt]
  \multirow{2}{*}{$\Bigl(\!$}
   & 3d+2  & w(*)  & 4  & 3  & 2  & 1  & 0
   & \multirow{2}{*}{$\!\Bigr),$} \\
   & -1 & 0 & 1  & 1  & 1  & 1  & 1 &
\end{array}
\end{equation}
where the part denoted by $w(*)$ is unchanged from \eqref{eq:wtmatdiv}, that is, it coincides with the weights of the original coordinates $*$ for $\Pi_{\mP}$. 
We can read off from \eqref{eq:wtmatdiv2} that the morphism $\Pi_{4}\to\overline{\Pi}_{5}$ is defined by
\begin{equation}\label{eq:divend}
\begin{aligned}
(w_{2},*,s_{246},s_{245},s_{136},s_{126},s_{124},s_{125},s_{135},s_{123})
&\mapsto\\
&(s_{123}^{3d+3}w_{2},\,s_{123}^{w(*)}*,\,s_{123}^{3}s_{246},\,s_{123}^{2}s_{245},\\
&\qquad s_{123}s_{136},\,s_{123}s_{126},\,s_{124},\,s_{125},\,s_{135}).
\end{aligned}
\end{equation}
Therefore, we see that the weights of the coordinates of the ambient weighted projective space of $\overline{\Pi}_{5}$ are given by the first row of the matrix \eqref{eq:wtmatdiv2}, excluding its rightmost entry. 
For the coordinates of this weighted projective space, we use the same notation as for the corresponding coordinates of $\Pi_{4}$. 
For example, the coordinate corresponding to $s_{123}^{3d+3}w_{2}$ is denoted simply by $w_{2}$. 

On the $s_{123}$-chart, after clearing denominators in \eqref{eq:s123}, we can eliminate the coordinates $w_{2}, p_{3}, t_{1}, t_{2}, u_{1}, u_{2}, s_{245}, s_{246}$ from the defining equations of the image $\overline{\Pi}_{5}$ of $\Pi_{4}$. 
Moreover, by using \eqref{eq:hypst}, we can also eliminate $t_{123}$. 
For example, multiplying both sides of the equation
\[
p_{3}=\tfrac{1}{s_{123}}\bigl(p_{1}s_{124}+p_{2}s_{126}-p_{4}s_{125}\bigr)
\]
in \eqref{eq:s123} by $s_{123}^{d+1}$, we obtain
\begin{align*}
s_{123}^{d+1}p_{3}
&=s_{123}^{d}\bigl(p_{1}s_{124}+p_{2}s_{126}-p_{4}s_{125}\bigr)\\
&=(s_{123}^{d}p_{1})s_{124}+(s_{123}^{d-1}p_{2})(s_{123}s_{126})-(s_{123}^{d}p_{4})s_{125},
\end{align*}
and hence the coordinate $p_{3}$ can be eliminated for the ambient weighted projective space $\overline{\Pi}_{5}$. 

Although this argument is carried out on the $s_{123}$-chart, it holds globally since $\Pi_{4}$ is irreducible. 
Therefore, we may regard $\overline{\Pi}_{5}$ as being contained in the weighted projective space whose coordinates are
\[
p_{1}, p_{2}, p_{4}, t_{124}, t_{125}, t_{126}, t_{135}, t_{136}, t_{245}, s_{124}, s_{125}, s_{126}, s_{135}, s_{136}.
\]
This weighted projective space, however, has the same dimension $13$ as $\overline{\Pi}_{5}$, and hence it coincides with $\overline{\Pi}_{5}$. 
Arranging the weights of these coordinates, we obtain exactly those given in the statement of the proposition.
\end{proof}

\begin{prop}
\label{prop:EndDiv}
The morphism $\,\widetilde{\Pi}\to w\mP$ is the contraction that contracts the divisor $\{s_{123}=0\}$ to the curve
\[
C:=\{s_{124}s_{135}=s_{125}^{2}\}\subset\mP(s_{124},s_{135},s_{125})\simeq\mP^{2}
\]
in $\,w\mP\simeq\mP(1^{4},2^{4},3^{3},d-1,d^{2})$. 
The exceptional divisor is a locally trivial $\,\mP(1^{2},2^{3},3^{3},4,d-1,d,d+1)$-bundle over this curve. 
In particular, $\,\widetilde{\Pi}$ has Picard number $\,2$.
\end{prop}

\begin{proof}
We study $\widetilde{\Pi}\to w\mP$ as $\Pi_{4}\to\overline{\Pi}_{5}$. 
It is the divisorial contraction whose exceptional divisor is $\{s_{123}=0\}$ due to the general theory of toric Sarkisov links.

If a point of the exceptional divisor satisfies $s_{124}=s_{135}=0$, then we see from the defining equations of $\Pi_{\mathtt{II}}$ that it also satisfies $s_{125}=0$. 
However, the corresponding point of $\Pi_{\mathtt{II}}$ is an unstable point for the GIT quotient defining $\Pi_{4}$. 
Therefore, every point on the exceptional divisor satisfies either $s_{124}\neq0$ or $s_{135}\neq0$, and hence it can be studied on the $s_{124}$-chart and the $s_{135}$-chart.

To determine the image of the divisor $\{s_{123}=0\}$, we work on the $s_{124}$-chart. Using \eqref{eq:divend} and setting $s_{123}=0$ in the equation
\[
s_{135}=\tfrac{1}{s_{124}^{2}}\bigl(s_{124}s_{125}^{2}+s_{123}s_{125}s_{126}+s_{123}^{2}s_{245}\bigr)
\]
appearing in \eqref{eq:s124}, we obtain exactly the description of the image of the divisor $\{s_{123}=0\}$ given in the statement of the proposition.

The structure of the exceptional divisor can be studied by using \eqref{eq:s135} and \eqref{eq:s124}. 
Subtracting $(3d+3)$ times the second row from the first row of the matrix \eqref{eq:wtmatdiv}, we obtain the following matrix:
\begin{equation}\label{eq:wtmatdiv3}
\setlength{\arraycolsep}{4pt}
\renewcommand{\arraystretch}{1.05}
\begin{array}{@{} c *{8}{c} c @{}}
  & w_2 & * & s_{246} & s_{245}
  & \substack{s_{126}\\s_{136}}
  & \substack{s_{124}\\s_{125}\\s_{135}}
  & s_{123} & \\[2pt]
  \multirow{2}{*}{$\Bigl(\!$}
   & 3d+3  & w(*)  & 3  & 2  & 1  & 0  & -1
   & \multirow{2}{*}{$\!\Bigr).$} \\
   & -1 & 0 & 1  & 1  & 1  & 1  & 1 &
\end{array}
\end{equation}
From this, we read off, in the same way as in Proposition~\ref{prop:transition-fcn}, that the exceptional divisor has the structure of a locally trivial $\mP(1^{2},2^{3},3^{3},4,d-1,d,d+1)$-bundle over the curve $C$.
\end{proof}

\subsection{Description of $\widetilde{g}_{\Pi}\colon\Pi_{1}\to\overline{\Pi}_{1}$}\label{subsec:Descriptions-of-flop}
Recall that $\widetilde{g}_{\Pi}\colon \Pi_{1}\to \overline{\Pi}_{1}$ is the morphism appearing in the diagrams \eqref{eq:SarkisovDelkey} and \eqref{eq:SarkisovDivkey}.

\begin{prop}
\label{prop:flop}It holds that 

\begin{enumerate}[$(1)$]

\item $\widetilde{g}_{\Pi}\colon\Pi_{1}\to\overline{\Pi}_{1}$ is
a crepant small contraction, and

\item both of the relative Picard numbers of $\ \widehat{g}_{\Pi}$
and $\,\widetilde{g}_{\Pi}$ are $1$, and $\, \widehat{\Pi}\dashrightarrow\Pi_{1}$
(resp.~$\, \widehat{\Pi}\dashrightarrow\widetilde{\Pi}$) is the flop
with respect to $-E_{\Pi}$ in the case of$\,$  {\rm{Type~Dp}} (resp.~{\rm{Type~Div}}). 

\end{enumerate}
\end{prop}

\begin{proof}
(1) By the discussion in the beginning of Subsection \ref{subsec:Constructing-strict-transforms Pi1}, we see that $\widetilde{g}_{\Pi}$ is surjective and birational. 
Moreover, by Proposition~\ref{prop:widehat=00007BPi=00007D bir}, $\overline{\Pi}_{1}$ is not $\mQ$-factorial.

In the case of Type~Dp, Propositions~\ref{prop:transition-fcn} and \ref{prop:flipping} show that the Picard number of $\Pi_{1}$ is $2$. 
Hence the relative Picard number of $\widetilde{g}_{\Pi}$ is $1$. 
Therefore, if $\widetilde{g}_{\Pi}$ were to contract a divisor, then $\overline{\Pi}_{1}$ would be $\mQ$-factorial, which is a contradiction. 
Thus $\widetilde{g}_{\Pi}$ is also a small contraction. 
Moreover, since $\overline{\Pi}_{1}$ is a hypersurface by Proposition \refeq {prop:widehat=00007BPi=00007D bir} and hence $\mQ$-Gorenstein, it follows that $\widehat{g}_{\Pi}$ is crepant.

In the case of Type~Div, by Proposition~\ref{eq:divend}, the Picard number of $\widetilde{\Pi}$ is $2$. 
The rest of the argument is the same as in the Type~Dp case.

\vspace{3pt}

$\noindent$(2) Since the proof is the same for both Type~Dp and Type~Div, we only argue in the case of Type~Dp. 
Note that the divisor $-E_{\widehat{\Pi}}$ is relatively ample with respect to $\widehat{f}_{\Pi}\colon \widehat{\Pi}\to \Pi_{\mP}$. 
Since $E_{\widehat{\Pi}}$ is effective and the Picard number of $\widehat{\Pi}$ is $2$, it follows that $E_{\widehat{\Pi}}$ is $\widehat{g}_{\Pi}$-ample. 
In particular, $E_{\overline{\Pi}}$ is not $\mQ$-Cartier. 
Moreover, by Lemma~\ref{lem:=00306F=009AD8=003005=005546=007279=007570=0070B9=003057=00304B=006301=00305F=00306A=003044=003002}, $\Pi_{1}$ is $\mQ$-factorial. 
Since the Picard number of $\Pi_{1}$ is $2$ by Proposition~\ref{prop:transition-fcn}, the strict transform $E_{\Pi_{1}}$ of $E_{\widehat{\Pi}}$ on $\Pi_{1}$, or its negative $-E_{\Pi_{1}}$, is $\widetilde{g}_{\Pi}$-ample. 
However, if $E_{\Pi_{1}}$ were $\widetilde{g}_{\Pi}$-ample, then $\widehat{\Pi}$ and $\Pi_{1}$ would be isomorphic, which is impossible since they admit different contractions (from $\widehat{\Pi}$ we have $\widehat{g}_{\Pi}$ and the weighted blow-up $\widehat{f}_{\Pi}\colon \widehat{\Pi}\to \Pi_{\mP}$, whereas from $\Pi_{1}$ we have $\widetilde{g}_{\Pi}$ and the flipping contraction; see Proposition~\ref{prop:flipping}). 
Therefore, $-E_{\Pi_{1}}$ is $\widetilde{g}_{\Pi}$-ample, and $\Pi_{0}\dashrightarrow \Pi_{1}$ is the flop with respect to $-E_{\Pi}$.
\end{proof}

\begin{proof}[\textbf{Proof of Theorem \ref{thm:keySarkisov}}]
In the case of Type~Dp, by gluing the diagrams \eqref{eq:firstres} and \eqref{eq:SarkisovDelkey}, and in the case of Type~Div, by gluing the diagrams \eqref{eq:firstres} and \eqref{eq:SarkisovDivkey}, and then combining the descriptions of these diagrams obtained above, we obtain the Sarkisov link as stated in Theorem \ref{thm:keySarkisov}.

It remains to prove that $\Pi_{\mP}$ has only terminal singularities. Since $\widehat{f}_{\Pi}$ is a $K$-negative divisorial contraction, it suffices to show that $\widehat{\Pi}$ has only terminal singularities. We argue separately for Type~Dp and Type~Div; however, since the proof proceeds in the same way in both cases, we treat only the case of Type~Dp. By Proposition \ref{prop:flop} and \cite{Ko}, it suffices to show that $\Pi_1$ has only terminal singularities. By Lemma \ref{lem:sijkNull}, it is enough to apply the Reid--Tai criterion (\cite[Thm.~4.11]{R}) on the $s_{124}$-, $s_{135}$-, $s_{123}$-, and $s_{246}$-charts.
\end{proof}

\section{\textbf{Sarkisov links for prime $\mQ$-Fano $3$-folds--Proof of Theorem~\ref{thm:FanoSarkisov}--}}
\label{sec:Proof-of-Theorem 3}

\subsection{$X$ is a prime $\mQ$-Fano $3$-fold}\label{subsection:prime}

In \cite{Tak6}, it was shown that, in each case for Type~Dp and Type~Div, a general weighted complete intersection in $\Pi_{\mP}$ is a prime $\mQ$-Fano $3$-fold. 
However, since the condition $(\star)$ in Theorem~\ref{thm:FanoSarkisov} is a Zariski closed condition, we must first verify that $X$ is a prime $\mQ$-Fano $3$-fold.

In the cases Nos.~501 and 512, the condition $(\star)$ is automatically satisfied for a general choice of sections to cut out $X$ from $\Pi_\mP$ due to the weights. 
Hence, by \cite[Thm.~1.2]{Tak6}, $X$ is a prime $\mQ$-Fano $3$-fold.

In the cases Nos.~550, 577, and 878, the defining equations of general sections  do not involve $s_{2}$ or $s_{3}$ due to the weights, but may involve $s_{1}$ (see \cite[Sec.~4]{Tak6}). 
However, using the automorphisms of $\Pi_{\mP}$ given in Proposition~\ref{prop:autom}, we can transform such an $X$ into one satisfying the condition $(\star)$. 
An explicit description of how this is done is given in Subsection~\ref{subsec:Explicite-equations-of=003000X} in each case. 
Therefore, also in these cases, $X$ is a prime $\mQ$-Fano $3$-fold by \cite[Thm.~1.2]{Tak6}.

In the remaining cases Nos.~872 and 1766, the defining equations of general sections do not involve $s_{3}$ due to the weights, but may involve either $s_{1}$ or $s_{2}$ (see \cite[Sec.~4]{Tak6}). 
Unfortunately, in these cases, no suitable automorphism of $\Pi_{\mP}$ is available so far, hence we must directly show that a general $X$ satisfying the condition $(\star)$ is a prime $\mQ$-Fano $3$-fold. 
Nevertheless, since the proof proceeds in the same way as in \cite{Tak6} and \cite{Tak7}, we omit it.

\vspace{5pt}

For the remaining parts of the proof, although the computations are elementary, some of them are difficult to carry out by hand. 
We therefore refer to \cite{Tak9} for the computational details and only present an outline of the proof in the sequel. 
Since the arguments for Type~Dp and Type~Div are similar in each case, \cite[Sec.~7 and 8]{Tak9} contains detailed computations only in the cases No.~501 (Type~Dp) and No.~577 (Type~Div).

\subsection{Description of $\widehat{f}$}

By the condition $(\star)$, we see that $X$ contains the $s_{1}$-point of $\Pi_{\mP}$. 
By intersecting the singularity \eqref{eq:s1-pt} of $\Pi_{\mP}$ at the $s_{1}$-point with sections as in \eqref{eq:TypeDp} and \eqref{eq:TypeDiv}, we expect numerically that $X$ has a $\nicefrac{1}{d+1}(1,1,d)$-singularity at the $s_{1}$-point. 
Using the explicit defining equations of $X$ given in Subsection~\ref{subsec:Explicite-equations-of=003000X}, we can verify that this is indeed the case. 
In the same way, we see that $\widehat{f}$ is the weighted blow-up with weights $\nicefrac{1}{d+1}(1,1,d)$. 
In particular, $\widehat{X}$ is $\mQ$-factorial and has Picard number $2$. 
The restriction $\widehat{X}\cap\{w_{1}=0\}$ of $E_{\widehat{\Pi}}$ to $\widehat{X}$ is the exceptional divisor of $\widehat{f}$, which we denote by $E_{\widehat{X}}$.

\subsection{Sarkisov link}

In the case of Type~Dp, let $X_{1}$ and $\widetilde{X}$ be the subvarieties of $\Pi_{1}$ and $\widetilde{\Pi}$, respectively, cut out by the sections defined by the same equations as those cutting out $X$ from $\Pi_{\mP}^{14}$ in Subsection~\ref{subsec:Explicite-equations-of=003000X}.
Thanks to the condition $(\star)$, this construction is well defined. 
By examining $X_{1}$ and $\widetilde{X}$ explicitly, we see that they are the strict transforms of $X$. 
Let $\overline{X}_{1}$ and $\mathring{X}$ be the images of $X_{1}$ under $\Pi_{1}\to\overline{\Pi}_{1}$ and $\Pi_{1}\to\mathring{\Pi}$, respectively. 
In this way, we obtain the restriction of the diagram obtained by concatenating \eqref{eq:firstres} and \eqref{eq:SarkisovDelkey}. 
We will show below that this is indeed the desired Sarkisov link. 
In the case of Type~Div, by performing the same construction for the diagrams \eqref{eq:firstres} and \eqref{eq:SarkisovDivkey}, we obtain the corresponding restricted diagram, which we will again show to be the desired Sarkisov link.

\subsection{Flop}

In the case of Type~Dp, we show that $\widehat{X}\to\overline{X}_{1}$ and $X_{1}\to\overline{X}_{1}$ are crepant small contractions, and that $\widehat{X}\dashrightarrow X_{1}$ is the $-E_{\widehat{X}}$-flop, as follows. 

Let $S$ be the anticanonical $K3$ surface of $X$. 
Note that, since $h^{0}(-K_{X})=1$ in all cases, $S$ is uniquely determined. 
Let $\sigma$ be one of the sections cutting out $S$ from $\Pi_{\mP}$, and let $\widehat{\sigma}$ be the corresponding section of $\widehat{\Pi}$ defined by the same equation. 
By the definition of the weighted blow-up $\widehat{f}_{\Pi}$, $\widehat{\sigma}$ is the strict transform of $\sigma$. 
Let $\widehat{S}$ be the subscheme of $\widehat{\Pi}$ cut out by the sections $\widehat{\sigma}$. 
Then the canonical divisor of $\widehat{S}$ is trivial. 
By examining the morphism $\widehat{S}\to S$ explicitly, we see that it is the weighted blow-up of type $\nicefrac{1}{d+1}(1,d)$. 
Hence $\widehat{S}$ is also an anticanonical $K3$ surface. 

Let $\overline{\sigma}$ be the section of $\overline{\Pi}_{1}$ defined by the same equation as $\sigma$, and let $\overline{S}_{1}$ be the scheme cut out from $\overline{\Pi}_{1}$ by these sections. 
Since $\widehat{\sigma}$ is the total transform of $\overline{\sigma}$ for each $\sigma$, the morphism $\widehat{g}_{\Pi}$ induces a surjective morphism $\widehat{S}\to\overline{S}_{1}$. 
By examining $\overline{S}_{1}$ explicitly via the Jacobian criterion, we see that its affine cone is smooth. 
Hence the affine cone of $\overline{S}_{1}$ is disjoint from the singular locus of the affine hypersurface $\{F_{\overline{\Pi}}=0\}$. 

If $\widehat{X}\to\overline{X}_{1}$ were a divisorial contraction, then the image of its exceptional divisor would be a curve since all its fibers are $1$-dimensional by Proposition~\ref{prop:widehat=00007BPi=00007D bir}. The curve would intersect $\overline{S}_{1}$ since $\overline{S}_{1}$ is ample. Since the curve is contained in the image of the exceptional locus of the small contraction $\widehat{g}_\Pi\colon \widehat{\Pi}\to \overline{\Pi}_1$, the singularities of $\overline{X}_1$ along the the curve is not produced by the finite group action.  
Hence the affine cone of $\overline{S}_{1}$ would intersect the singular locus of the affine hypersurface $\{F_{\overline{\Pi}}=0\}$, a contradiction. 
Therefore, $\widehat{X}\to\overline{X}_{1}$ is a crepant small contraction. 
In particular, $\overline{X}_{1}$ has only terminal singularities. 
It follows that $X_{1}\to\overline{X}_{1}$ is also a crepant small contraction. 
Since the Picard number of $\widehat{X}$ is $2$, and both $\widehat{X}\to\overline{X}_{1}$ and $X_{1}\to\overline{X}_{1}$ are crepant small contractions, the Picard number of $X_{1}$ is also $2$. 
The proof is then completed in the same way as in Proposition~\ref{prop:flop}.

In the case of Type~Div, the same statements hold for $\widehat{X}\to\overline{X}_{1}$ and $\widetilde{X}\to\overline{X}_{1}$, and can be proved in the same way.

\vspace{5pt}

The remaining proof of Theorem~\ref{thm:FanoSarkisov} is carried out separately for Type~Dp and Type~Div.

\subsection{Type~Dp}

\subsubsection{\textbf{Flip}}

As in the proof of Theorem~\ref{thm:keySarkisov}, it suffices to work on the $s_{124}$-chart. 
By cutting the toric flip described in Theorem~\ref{thm:keySarkisov} by the sections corresponding to \eqref{eq:TypeDp}, we expect numerically that it becomes a hypersurface flip of type $(d,1^{2},-2,-1;d-2)$. 
Using the explicit defining equations of $X$ given in Subsection~\ref{subsec:Explicite-equations-of=003000X}, we verify that this is indeed the case. 
It is also easy to see from the explicit description that the exceptional locus of $X_{1}\to\mathring{X}$ is as described in Theorem~\ref{thm:FanoSarkisov}. 

The only slightly delicate point is to show that the Gorenstein singularity on the exceptional locus of $\widetilde{X}\to\mathring{X}$ is an ordinary $cA_{d-3}$ singularity (it is easy to see that there are one Gorenstein  singularity and one $\nicefrac{1}{2}(1,1,1)$-singularity on this locus). 
We briefly explain this point.

It is straightforward to locate the Gorenstein singularities by the Jacobian criterion. 
Near such a point, $X_{1}$ is a hypersurface in $\mC^{4}$ with coordinates $s_{123},p_{3},s_{126},z$. 
To translate the singular point to the origin, it is necessary to change only the coordinate $s_{123}$; denote the new coordinate by $G_{123}$. For No. 872, since the rank of the quadratic part of this hypersurface at the origin is seen to be 4, the singularity in question is an ordinary double point as desired. 
In the other cases, the lowest-degree term of the defining equation of $X_{1}$ is $G_{123}p_{3}$. 
Assigning weights
\[
\begin{cases}
(2,3,1,1): & {\rm No}.\,501,\\
(1,3,1,1): & {\rm No}.\,512,\\
(1,2,1,1): & {\rm No}.\,550
\end{cases}
\]
to $(G_{123},p_{3},s_{126},z)$, the leading part of $X_{1}$ with respect to this weighting is the sum of $G_{123}p_{3}$ and a polynomial $h(s_{126},z)$ depending only on $s_{126}$ and $z$, without multiple factors. 
Therefore, by Lemma~\ref{lem:canonical form}, it follows that the Gorenstein singularity is an ordinary $cA_{d-3}$-singularity.

\subsubsection{\textbf{Del Pezzo fibration}}

It suffices to work on the $s_{135}$-chart or the $s_{123}$-chart. 
The fibers of $\widetilde{\Pi}\to\mP^{1}$ are weighted projective spaces as described in Theorem~\ref{thm:keySarkisov} (see also Proposition~\ref{prop:transition-fcn}). 
By cutting them with the sections corresponding to \eqref{eq:TypeDp}, we expect numerically to obtain a del Pezzo surface of degree $1$, namely $(6)\subset\mP(1^{2},2,3)$. 
Using the explicit defining equations of $X$ given in Subsection~\ref{subsec:Explicite-equations-of=003000X}, we verify that this is indeed the case for the general fiber. 
Therefore, $\widetilde{f}\colon\widetilde{X}\to\mP^{1}$ is a fibration whose general fiber is a del Pezzo surface of degree $1$.

This completes the proof of Theorem~\ref{thm:FanoSarkisov} in the case of Type~Dp.

\subsection{Type~Div}

That $Y$ is a weighted complete intersection as described in Theorem~\ref{thm:FanoSarkisov} follows numerically from the fact that $w\mP$ is the weighted projective space described in Theorem~\ref{thm:keySarkisov} (see also Proposition~\ref{prop:Pi5WPS}), and that $Y$ is obtained by cutting it with the sections corresponding to \eqref{eq:TypeDiv}. 
Using the explicit defining equations of $X$ given in Subsection~\ref{subsec:Explicite-equations-of=003000X}, we verify that this is indeed the case. 

It follows easily that $\widetilde{f}\colon\widetilde{X}\to Y$ contracts a divisor to a curve $C$, since $\widetilde{f}_{\Pi}\colon\widetilde{\Pi}\to w\mP$ has the same property. 
Since $\widehat{X}\dashrightarrow\widetilde{X}$ is a flop, the singularities of $\widetilde{X}$ coincide with those of $\widehat{X}$ by \cite{Ko}. 

It remains to show that, near each singular point of $Y$ along $C$, the morphism $\widetilde{f}$ is a simple $(2,1)$-contraction. 
Near the exceptional divisor $\{s_{123}=0\}$, the variety $\widetilde{X}$ is covered by the union of the $s_{124}$-, $s_{125}$-, and $s_{135}$-charts, as follows from the description of the unstable locus of the GIT quotient. 
Thus it suffices to work on these charts, and in fact it is enough to work on the $s_{124}$-chart, since all singular points of $\widetilde{X}$ lying on the $\widetilde{f}$-exceptional divisor turn out to be contained in this chart. We can check explicitly the conditions of Proposition~\ref{prop:simple cont} near the fibers through singular points of $\widetilde{X}$ lying on the $\widetilde{f}$-exceptional divisor, and hence $\widetilde{f}$ is a simple $(2,1)$-contraction near such fibers.

This completes the proof of Theorem~\ref{thm:FanoSarkisov} in the case of Type~Div.

\section*{\textbf{Appendix}}

\renewcommand{\thethm}{A.\arabic{thm}}
\renewcommand{\theequation}{A.\arabic{equation}}

\setcounter{thm}{0}
\setcounter{equation}{0}

\renewcommand{\thesubsection}{A.\arabic{subsection}}
\setcounter{subsection}{0}

\renewcommand{\thethm}{A.\arabic{thm}}
\renewcommand{\theequation}{A.\arabic{equation}}

\setcounter{thm}{0}
\setcounter{equation}{0}
\subsection{Equation of $\widehat{\Pi}$}\label{subsec:The-equation-ofPi0}

The following polynomials $F_1$--$F_9$ define the variety $\Pi_0$. Note that, setting $w_{2}=1$ in the following equations, we recover the defining equations of $\Pi_{\mP}$ (cf.~\cite[A.2]{Tak6}).

\begin{align*}
F_1
&=
-p_1 s_2 - p_2 s_3 + s_1 u_1
\notag\\
&\quad
+ \Bigl(
- p_3^2
- p_4^2 t_1
+ p_2 p_3 t_{126}
- p_1 p_4 t_{126}
- p_1^2 t_{136}
- p_2^2 t_1 t_{136}
- 2 p_1 p_3 t_{245}
\notag\\
&\qquad\qquad
- 2 p_2 p_4 t_1 t_{245}
- p_2 t_{136} u_1
- p_4 t_{245} u_1
- p_4 u_2
- p_2 t_{245} u_2
\Bigr) w_2,
\\[1ex]
F_2
&=
-p_3 s_2 - p_4 s_3 + s_1 u_2
\notag\\
&\quad
+ \Bigl(
- p_2 p_3 t_{125}
+ p_1 p_4 t_{125}
+ p_1^2 t_{135}
+ p_2^2 t_1 t_{135}
+ 2 p_1 p_3 t_{136}
+ 2 p_2 p_4 t_1 t_{136}
\notag\\
&\qquad\qquad
+ p_3^2 t_{245}
+ p_4^2 t_1 t_{245}
+ p_2 t_{135} u_1
+ p_4 t_{136} u_1
+ p_2 t_{136} u_2
+ p_4 t_{245} u_2
\Bigr) w_2,
\\[1ex]
F_3
&=
-p_1 s_3 + p_2 s_2 t_1 - p_2 s_1 t_2 + s_2 u_1
\notag\\
&\quad
+ \Bigl(
p_2 p_3 t_{124}
- p_1 p_4 t_{124}
- p_4^2 t_2
- p_2^2 t_{136} t_2
- 2 p_2 p_4 t_2 t_{245}
- p_1 t_{136} u_1
\notag\\
&\qquad\qquad
- p_3 t_{245} u_1
- p_3 u_2
- p_1 t_{245} u_2
\Bigr) w_2,
\\[1ex]
F_4
&=
-p_3 s_3 + p_4 s_2 t_1 - p_4 s_1 t_2 + s_2 u_2
\notag\\
&\quad
+ \Bigl(
- p_2 p_3 t_{123}
+ p_1 p_4 t_{123}
+ p_2^2 t_{135} t_2
+ 2 p_2 p_4 t_{136} t_2
+ p_4^2 t_2 t_{245}
\notag\\
&\qquad\qquad
+ p_1 t_{135} u_1
+ p_3 t_{136} u_1
+ p_1 t_{136} u_2
+ p_3 t_{245} u_2
\Bigr) w_2,
\\[1ex]
F_5
&=
p_1 s_2 t_1 + p_2 s_3 t_1 - p_1 s_1 t_2 - p_2 s_2 t_2 + s_3 u_1
\notag\\
&\quad
+ \Bigl(
p_1^2 t_{123}
+ p_2^2 t_1 t_{123}
+ p_1 p_3 t_{124}
+ p_2 p_4 t_1 t_{124}
+ p_3 p_4 t_2
- p_2^2 t_{125} t_2
- p_2 p_4 t_{126} t_2
\notag\\
&\qquad\qquad
+ p_1 p_2 t_{136} t_2
+ p_2 p_3 t_2 t_{245}
+ p_1 p_4 t_2 t_{245}
+ p_2 t_{123} u_1
- p_1 t_{125} u_1
- p_2 t_1 t_{136} u_1
- p_4 t_1 t_{245} u_1
\notag\\
&\qquad\qquad
- t_{136} u_1^2
- p_4 t_1 u_2
+ p_2 t_{124} u_2
- p_1 t_{126} u_2
- p_2 t_1 t_{245} u_2
- 2 t_{245} u_1 u_2
- u_2^2
\Bigr) w_2,
\\[1ex]
F_6
&=
p_3 s_2 t_1 + p_4 s_3 t_1 - p_3 s_1 t_2 - p_4 s_2 t_2 + s_3 u_2
\notag\\
&\quad
+ \Bigl(
p_1 p_3 t_{123}
+ p_2 p_4 t_1 t_{123}
+ p_3^2 t_{124}
+ p_4^2 t_1 t_{124}
- p_2 p_4 t_{125} t_2
- p_4^2 t_{126} t_2
\notag\\
&\qquad\qquad
- p_1 p_2 t_{135} t_2
- p_2 p_3 t_{136} t_2
- p_1 p_4 t_{136} t_2
- p_3 p_4 t_2 t_{245}
+ p_4 t_{123} u_1
- p_3 t_{125} u_1
\notag\\
&\qquad\qquad
+ p_2 t_1 t_{135} u_1
+ p_4 t_1 t_{136} u_1
+ t_{135} u_1^2
+ p_4 t_{124} u_2
- p_3 t_{126} u_2
+ p_2 t_1 t_{136} u_2
\notag\\
&\qquad\qquad
+ p_4 t_1 t_{245} u_2
+ 2 t_{136} u_1 u_2
+ t_{245} u_2^2
\Bigr) w_2,
\\[1ex]
F_7
&=
-s_2^2 + s_1 s_3
\notag\\
&\quad
+ \Bigl(
p_2 s_1 t_{123}
+ p_4 s_1 t_{124}
- p_2 s_2 t_{125}
- p_4 s_2 t_{126}
\Bigr) w_2
\notag\\
&\quad
+ \Bigl(
p_4^2 t_{123}
- p_3 p_4 t_{125}
- p_1 p_3 t_{135}
+ p_2 p_4 t_1 t_{135}
- p_2^2 t_{124} t_{135}
+ p_1 p_2 t_{126} t_{135}
- p_3^2 t_{136}
\notag\\
&\qquad\qquad
+ p_4^2 t_1 t_{136}
+ p_2^2 t_{123} t_{136}
- 2 p_2 p_4 t_{124} t_{136}
- p_1 p_2 t_{125} t_{136}
+ p_2 p_3 t_{126} t_{136}
+ p_1 p_4 t_{126} t_{136}
\notag\\
&\qquad\qquad
+ p_1^2 t_{136}^2
- p_2^2 t_1 t_{136}^2
+ 2 p_2 p_4 t_{123} t_{245}
- p_4^2 t_{124} t_{245}
- p_2 p_3 t_{125} t_{245}
- p_1 p_4 t_{125} t_{245}
\notag\\
&\qquad\qquad
+ p_3 p_4 t_{126} t_{245}
- p_1^2 t_{135} t_{245}
+ p_2^2 t_1 t_{135} t_{245}
+ p_1 p_3 t_{136} t_{245}
- p_2 p_4 t_1 t_{136} t_{245}
+ p_3^2 t_{245}^2
\notag\\
&\qquad\qquad
- p_4^2 t_1 t_{245}^2
+ p_4 t_{135} u_1
- 2 p_2 t_{136}^2 u_1
+ 2 p_2 t_{135} t_{245} u_1
- p_4 t_{136} t_{245} u_1
+ p_2 t_{135} u_2
\notag\\
&\qquad\qquad
+ 2 p_4 t_{136} u_2
- p_2 t_{136} t_{245} u_2
- 2 p_4 t_{245}^2 u_2
\Bigr) w_2^2,
\\[1ex]
F_8
&=
s_2 s_3 + s_1 s_2 t_1 - s_1^2 t_2
\notag\\
&\quad
+ \Bigl(
p_1 s_1 t_{123}
+ p_3 s_1 t_{124}
- p_1 s_2 t_{125}
- p_3 s_2 t_{126}
\Bigr) w_2
\notag\\
&\quad
+ \Bigl(
p_3 p_4 t_{123}
- p_3^2 t_{125}
+ p_2 p_3 t_1 t_{135}
+ p_1 p_4 t_1 t_{135}
- p_1 p_2 t_{124} t_{135}
+ p_1^2 t_{126} t_{135}
\notag\\
&\qquad\qquad
+ 2 p_3 p_4 t_1 t_{136}
+ p_1 p_2 t_{123} t_{136}
- p_2 p_3 t_{124} t_{136}
- p_1 p_4 t_{124} t_{136}
- p_1^2 t_{125} t_{136}
+ 2 p_1 p_3 t_{126} t_{136}
\notag\\
&\qquad\qquad
- 2 p_1 p_2 t_1 t_{136}^2
- p_2 p_4 t_{135} t_2
- p_4^2 t_{136} t_2
+ p_2^2 t_{136}^2 t_2
+ p_2 p_3 t_{123} t_{245}
+ p_1 p_4 t_{123} t_{245}
\notag\\
&\qquad\qquad
- p_3 p_4 t_{124} t_{245}
- 2 p_1 p_3 t_{125} t_{245}
+ p_3^2 t_{126} t_{245}
+ 2 p_1 p_2 t_1 t_{135} t_{245}
- p_2 p_3 t_1 t_{136} t_{245}
- p_1 p_4 t_1 t_{136} t_{245}
\notag\\
&\qquad\qquad
- p_2^2 t_{135} t_2 t_{245}
+ p_2 p_4 t_{136} t_2 t_{245}
- 2 p_3 p_4 t_1 t_{245}^2
+ p_4^2 t_2 t_{245}^2
+ p_3 t_{135} u_1
- 2 p_1 t_{136}^2 u_1
\notag\\
&\qquad\qquad
+ 2 p_1 t_{135} t_{245} u_1
- p_3 t_{136} t_{245} u_1
+ p_1 t_{135} u_2
+ 2 p_3 t_{136} u_2
- p_1 t_{136} t_{245} u_2
- 2 p_3 t_{245}^2 u_2
\Bigr) w_2^2,
\\[1ex]
F_9
&=
s_3^2 + s_2^2 t_1 - s_1 s_2 t_2
\notag\\
&\quad
+ \Bigl(
s_1 t_{123} u_1
- s_2 t_{125} u_1
+ s_1 t_{124} u_2
- s_2 t_{126} u_2
\Bigr) w_2
\notag\\
&\quad
+ \Bigl(
- p_3^2 t_{123}
- p_4^2 t_1 t_{123}
- p_2 p_3 t_{124} t_{125}
+ p_1 p_4 t_{124} t_{125}
+ p_2 p_3 t_{123} t_{126}
- p_1 p_4 t_{123} t_{126}
\notag\\
&\qquad\qquad
+ p_1^2 t_{124} t_{135}
+ p_2^2 t_1 t_{124} t_{135}
- p_1^2 t_{123} t_{136}
- p_2^2 t_1 t_{123} t_{136}
+ 2 p_1 p_3 t_{124} t_{136}
+ 2 p_2 p_4 t_1 t_{124} t_{136}
\notag\\
&\qquad\qquad
+ p_4^2 t_{125} t_2
+ p_2 p_3 t_{135} t_2
+ p_1 p_4 t_{135} t_2
- p_2^2 t_{126} t_{135} t_2
+ 2 p_3 p_4 t_{136} t_2
+ p_2^2 t_{125} t_{136} t_2
\notag\\
&\qquad\qquad
- 2 p_2 p_4 t_{126} t_{136} t_2
- 2 p_1 p_2 t_{136}^2 t_2
- 2 p_1 p_3 t_{123} t_{245}
- 2 p_2 p_4 t_1 t_{123} t_{245}
+ p_3^2 t_{124} t_{245}
+ p_4^2 t_1 t_{124} t_{245}
\notag\\
&\qquad\qquad
+ 2 p_2 p_4 t_{125} t_2 t_{245}
- p_4^2 t_{126} t_2 t_{245}
+ 2 p_1 p_2 t_{135} t_2 t_{245}
- p_2 p_3 t_{136} t_2 t_{245}
- p_1 p_4 t_{136} t_2 t_{245}
\notag\\
&\qquad\qquad
- 2 p_3 p_4 t_2 t_{245}^2
- t_{136}^2 u_1^2
+ t_{135} t_{245} u_1^2
+ t_{135} u_1 u_2
- t_{136} t_{245} u_1 u_2
+ t_{136} u_2^2
- t_{245}^2 u_2^2
\Bigr) w_2^2.
\end{align*}

\subsection{Automorphisms of $\Pi_{\mP}$}\label{subsec:Automorphisms-of}
\begin{prop}
\label{prop:autom}
The following assertions hold:
\begin{enumerate}[$(1)$]

\item Let $\Pi_{\mP}$ be a key variety for \textup{Type Div}. There exists an automorphism of $\,\Pi_{\mP}$ which moves a point in $\Pi_{\mP}\cap\mP(p_{3},u_{1},t_{123},s_1)$
with $u_{1}=p_{3}^2,\,s_{1}=1$ (resp. a point in $\Pi_{\mP}\cap\mP(p_{3},u_{1},s_{1})$ with $u_{1}=p_{3}^2,\,s_{1}=1$) to a point in $\Pi_{\mP}\cap\mP(t_{123},s_1)$
with $s_{1}=1$ (resp.  the $s_{1}$-point).

\item Let $\Pi_{\mP}$ be the key variety for No.550 or 878. There exists an automorphism of $\,\Pi_{\mP}$
which moves a point in $\Pi_{\mP}\cap \mP(t_{123},t_{135},s_1)$
with $s_{1}=1$ (resp. a point in $\Pi_{\mP}\cap \mP(t_{123},s_1)$ with $s_1=1$) to a point in $\Pi_{\mP}\cap \mP(t_{135},s_1)$
with $s_{1}=1$ (resp. the $s_1$-point).

\item Let $\Pi_{\mP}$ be the key variety for No.550. There exists an automorphism of $\,\Pi_{\mP}$
which moves a point in $\mP(t_{135},s_1)$
with $s_{1}=1$ to the $s_1$-point.
\end{enumerate}
\end{prop}

\begin{proof}
See \cite[Sec.~6]{Tak9}.
\end{proof}

\subsection{Explicit equations of prime $\mQ$-Fano $3$-folds}\label{subsec:Explicite-equations-of=003000X}

In this subsection, we record explicit data of the sections cutting out $X$ from $\Pi_{\mP}$. 
In \cite[Sec.~4]{Tak6}, sections defining the anticanonical $K3$ surface of $X$ are given explicitly. 
Starting from those, we slightly modify the choice of parameters so that, in each case of Type~Dp and Type~Div, the coordinates of the ambient weighted projective space of $X$ agree. 
(For Type~Dp No.~872, we need to replace $p_{2}$ by $z$.) 
For convenience, we list the resulting data here.

For Nos.~550, 577, and 878, we also explain how to apply Proposition~\ref{prop:autom} in the proof of Theorem~\ref{thm:FanoSarkisov} (Subsection \ref{subsection:prime}).

\subsubsection{\textbf{\textup{Type Dp}}}

\noindent{\small{}}{\small\par}

\noindent\vspace{3pt}

\noindent{}%
\fcolorbox{black}{white}{\textbf{No.$\,$501}}

\vspace{3pt}

\noindent\textbf{$\bullet$ Sections for $X$} 

\renewcommand{\arraystretch}{1.4}
\[
\begin{array}{|c|c|c|c|}
\hline
\text{weight} & 2 & 3 & 4\\
\hline	
\text{section}
&
\begin{array}{c}
t_{1}=a_{1}z^{2}\\
t_{245}=a_{245}z^{2}
\end{array}
&
\begin{array}{c}
t_{2}=a_{2}t_{126}\\
\quad +\,b_{2}z^{3}
\end{array}
&
\begin{array}{c}
p_{2}=c_{2}t_{126}z+d_{2}z^{4}\\
t_{124}=a_{124}t_{126}z+b_{124}z^{4}\\
t_{136}=a_{136}t_{126}z+b_{136}z^{4}
\end{array}
\\
\hline
\end{array}
\]

\[
\renewcommand{\arraystretch}{1.4}
\begin{array}{|c|c|c|}
\hline
\text{weight} & 5 & 6\\
\hline
\text{section}
&
\begin{array}{c}
p_{1}=b_1 t_{126} z^2+c_1 z^5\\
t_{125}=a_{125} t_{126} z^2+b_{125} z^5
\end{array}
&
\begin{array}{c}
p_{4}=a_{4}u_{1}+b_{4}t_{126}^{2}+c_4 t_{126} z^3+d_4 z^6\\
t_{123}=a_{123}u_{1}+b_{123}t_{126}^{2}+c_{123} t_{126} z^3+d_{123} z^6\\
t_{135}=a_{135}u_{1}+b_{135}t_{126}^{2}+c_{135} t_{126} z^3+d_{135} z^6
\end{array}
\\
\hline
\end{array}
\]

\noindent\textbf{$\bullet$ Embedding of $X$}: 
\[
X\subset\mP(z,t_{126},u_{1},p_{3},u_{2},s_{1},s_{2},s_{3})=\mP(1,3,6,7,8^{2},9,10).
\]

\vspace{5pt}

\noindent{}%
\fcolorbox{black}{white}{\textbf{No.$\,$512}}

\vspace{3pt}

\noindent\textbf{$\bullet$ Sections for $X$} { {\small{}}{\small\par}

\renewcommand{\arraystretch}{1.4}
\[
\begin{array}{|c|c|c|}
\hline
\text{weight} & 2 & 3\\
\hline
\text{section}
&
\begin{array}{c}
t_{1}=a_1 z^2\\
t_{245}=a_{245} z^2
\end{array}
&
\begin{array}{c}
p_{2}=a_{2}t_{126}+b_2 z^3\\
t_{2}=c_{2}t_{126}+d_2 z^3
\end{array}
\\
\hline
\end{array}
\]

\[
\renewcommand{\arraystretch}{1.4}
\begin{array}{|c|c|c|}
\hline
\text{weight} & 4 & 5\\
\hline
\text{section}
&
\begin{array}{c}
p_{1}=b_1 t_{126} z+ c_1 z^4\\
t_{124}=a_{124} t_{126} z+b_{124} z^4\\
t_{136}=a_{136} t_{126} z+b_{136} z^4
\end{array}
&
\begin{array}{c}
p_{4}=a_{4}u_{1}+b_4 t_{126}z^2+c_4 z^5\\
t_{125}=a_{125}u_{1}+b_{125} t_{126} z^2+c_{125} z^5
\end{array}
\\
\hline
\end{array}
\]

\[
\renewcommand{\arraystretch}{1.4}
\begin{array}{|c|c|}
\hline
\text{weight} & 6\\
\hline
\text{section}
&
\begin{array}{c}
t_{123}=a_{123}p_{3}+b_{123}t_{126}^{2}+c_{123} u_1 z+d_{123} t_{126} z^3+e_{123}z^6\\
t_{135}=a_{135}p_{3}+b_{135}t_{126}^{2}+c_{135} u_1 z+d_{135} t_{126} z^3+e_{135}z^6
\end{array}
\\
\hline
\end{array}
\]

\noindent\textbf{$\bullet$ Embedding of $X$}: 
\[
X\subset\mP(z,t_{126},u_{1},p_{3},u_{2},s_{1},s_{2},s_{3})=\mP(1,3,5,6,7^{2},8,9).
\]

\vspace{5pt}

\noindent{}%
\fcolorbox{black}{white}{\textbf{No.$\,$550}}

\vspace{3pt}

\noindent\textbf{$\bullet$ Sections for $X$} 

\renewcommand{\arraystretch}{1.4}
\[
\begin{array}{|c|c|c|}
\hline
\text{weight} & 2 & 3\\
\hline
\text{section}
&
\begin{array}{c}
p_{2}=a_2 z^2\\
t_{1}=a_1 z^2\\
t_{245}=a_{245}z^2
\end{array}
&
\begin{array}{c}
t_2=b_2 t_{126}+c_2 z^3\\
p_1=b_{1}t_{126}+c_1 z^3
\end{array}
\\
\hline
\end{array}
\]

\[
\renewcommand{\arraystretch}{1.4}
\begin{array}{|c|c|c|}
\hline
\text{weight} & 4 & 5\\
\hline
\text{section}
&
\begin{array}{c}
p_{4}=a_{4}u_{1}+b_4 t_{126} z+c_4 z^4\\
t_{124}=a_{124}u_{1}+b_{124} t_{126} z+c_{124} z^4\\
t_{136}=a_{136}u_{1}+b_{136} t_{126} z+c_{136} z^4
\end{array}
&
\begin{array}{c}
t_{125}=a_{125}p_{3}\\
\quad + b_{125}u_{1}z\\
\quad + c_{125} t_{126} z^2\\
\quad + d_{125} z^5
\end{array}
\\
\hline
\end{array}
\]
\[
\renewcommand{\arraystretch}{1.4}
\begin{array}{|c|c|}
\hline
\text{weight} & 6\\
\hline
\text{section}
&
\begin{array}{c}
t_{123}=a_{123}u_{2}+b_{123}t_{126}^{2}+c_{123}p_{3}z+d_{123}u_{1}z^2+e_{123} t_{126} z^3+f_{123} z^6\\
t_{135}=a_{135}u_{2}+b_{135}t_{126}^{2}+c_{135}p_{3}z+d_{135}u_{1}z^2+e_{135} t_{126} z^3+f_{135} z^6
\end{array}
\\
\hline
\end{array}
\]

\noindent\textbf{$\bullet$ Embedding of $X$}: 
\[
X\subset\mP(z,t_{126},u_{1},p_{3},u_{2},s_{1},s_{2},s_{3})=\mP(1,3,4,5,6^{2},7,8).
\]
\vspace{3pt}
\noindent For a more general $X$, the two sections of weight $6$ is obtained by adding terms $g_{123}s_{1}$ and $g_{135}s_1$ (with $g_{123},g_{135} \in \mC$) to the right-hand sides of the corresponding expressions listed in the above table. In this situation, let $\mathsf{p}_{6}$ denote the $\nicefrac{1}{6}(1,1,5)$-singularity of $X$. We see that this point lies in $\Pi_\mP\cap \mP(u_2,t_{123},t_{135},s_1)=\mP(t_{123},t_{135},s_1)$. 
Since $X$ is general, we may assume that $\mathsf{p}_{6}$ satisfies $s_{1}=1$. Then, by applying the automorphism of $\Pi_{\mP}$ described in Proposition~\ref{prop:autom}~(2), we may move $\mathsf{p}_{6}$ to a point in $\mP(t_{135},s_1)$ with $s_1=1$. Moreover, the automorphism as in Proposition \ref{prop:autom} (3) moves
a point in $\mP(t_{135},s_1)$ with $s_1=1$ to the $s_1$-point.
Therefore, we may assume that $\mathsf{p}_{6}$ is the $s_{1}$-point of $\Pi_{\mP}$. Under this assumption, it follows that $g_{123}=g_{135}=0$. Consequently, the variety $X$ defined by the sections given in the above table is also general.
\vspace{5pt}

\noindent{}%
\fcolorbox{black}{white}{\textbf{No.$\,$872}}\textbf{ }

\vspace{3pt}

\noindent\textbf{$\bullet$ Sections for $X$} 

\renewcommand{\arraystretch}{1.4}
\[
\begin{array}{|c|c|c|}
\hline
\text{weight} & 2 & 3\\
\hline
\text{section}
&
\begin{array}{c}
p_{1}=a_{1}p_{2}^{2}\\
t_{1}=b_{1}p_{2}^{2}\\
t_{245}=a_{245}p_{2}^{2}
\end{array}
&
\begin{array}{c}
t_{2}=a_{2}t_{126}+b_{2}u_{1}+c_{2}p_{2}^{3}\\
p_4=a_4 t_{126}+b_4 u_{1}+c_4 p_{2}^{3}
\end{array}
\\
\hline
\end{array}
\]

\renewcommand{\arraystretch}{1.4}
\[
\begin{array}{|c|c|}
\hline
\text{weight} & 4\\
\hline
\text{section}
&
\begin{array}{c}
t_{124}=a_{124}p_{3}+b_{124}t_{126} p_{2}+c_{124}u_{1}p_{2}+d_{124}p_{2}^{4}\\
t_{136}=a_{136}p_{3}+b_{136}t_{126} p_{2}+c_{136}u_{1}p_{2}+d_{136}p_{2}^{4}
\end{array}
\\
\hline
\end{array}
\]

\renewcommand{\arraystretch}{1.4}
\[
\begin{array}{|c|c|}
\hline
\text{weight} & 5\\
\hline
\text{section}
&
\begin{array}{c}
t_{125}=a_{125}u_{2}+b_{125}p_{3}p_{2}+c_{125}t_{126} p_{2}^{2}\\
\qquad{}+d_{125}u_{1}p_{2}^{2}+e_{125}p_{2}^{5}
\end{array}
\\
\hline
\end{array}
\]

\renewcommand{\arraystretch}{1.4}
\[
\begin{array}{|c|c|}
\hline
\text{weight} & 6\\
\hline
\text{section}
&
\begin{array}{c}
t_{123}=a_{123}t_{126}^{2}+b_{123}t_{126} u_{1}+c_{123}u_{1}^{2}+d_{123}u_{2}p_{2}\\
\qquad{}+e_{123}p_{3}p_{2}^{2}+f_{123}t_{126} p_{2}^{3}+g_{123}u_{1}p_{2}^{3}+h_{123}p_{2}^{6}\\[2pt]
t_{135}=a_{135}t_{126}^{2}+b_{135}t_{126}u_{1}+c_{135}u_{1}^{2}+d_{135}u_{2}p_{2}\\
\qquad{}+e_{135}p_{3}p_{2}^{2}+f_{135}t_{126} p_{2}^{3}+g_{135}u_{1}p_{2}^{3}+h_{135}p_{2}^{6}
\end{array}
\\
\hline
\end{array}
\]

\noindent\textbf{$\bullet$ Embedding of }$X$: 
\[
X\subset\mP(p_{2},t_{126},u_{1},p_{3},u_{2},s_{1},s_{2},s_{3})=\mP(1,3^{2},4,5^{2},6,7).
\]

\subsubsection{\textbf{\textup{Type Div}}}

\noindent{\small{}}{\small\par}

\noindent\vspace{3pt}

\noindent{}%
\fcolorbox{black}{white}{\textbf{No.$\,$577}}

\vspace{3pt}

\noindent\textbf{$\bullet$ Sections for $X$}

\renewcommand{\arraystretch}{1.4}
\[
\begin{array}{|c|c|c|}
\hline
\text{weight} & 2 & 3\\
\hline
\text{section}
&
\begin{array}{c}
t_{1}=a_{1}t_{245}^{2}\\
t_{126}=a_{126}t_{245}^{2}\\
t_{136}=a_{136}t_{245}^{2}
\end{array}
&
\begin{array}{c}
p_2=a_2 t_{135}+b_2 t_{245}^{3}\\
t_{2}=c_{2}t_{135}+d_{2}t_{245}^{3}\\
t_{124}=a_{124}t_{135}+b_{124}t_{245}^{3}\\
t_{125}=a_{125}t_{135}+b_{125}t_{245}^{3}
\end{array}
\\
\hline
\end{array}
\]

\[
\renewcommand{\arraystretch}{1.4}
\begin{array}{|c|c|c|}
\hline
\text{weight} & 4 & 5\\
\hline
\text{section}
&
\begin{array}{c}
p_{4}=a_{4}p_{1}+b_{4}t_{135}t_{245}+c_{4}t_{245}^{4}\\
t_{123}=a_{123}p_{1}+b_{123}t_{135}t_{245}+c_{123}t_{245}^{4}
\end{array}
&
\begin{array}{c}
\begin{aligned}
p_{3} &= a_{3}u_{1}\\
&\quad + b_{3}p_{1}t_{245}\\
&\quad + c_{3}t_{135}t_{245}^{2}\\
&\quad + d_{3}t_{245}^{5}
\end{aligned}
\end{array}
\\
\hline
\end{array}
\]
\noindent$\bullet$\textbf{ Embedding of $X$}: 
\[
X\subset\mP(t_{245},t_{135},p_{1},u_{1},s_{1},u_{2},s_{2},s_{3})=\mP(1,3,4,5^{2},6^{2},7).
\]

\noindent\vspace{3pt}

\noindent 

For a more general $X$, the section of weight $5$ is obtained by adding a term $e_{3}s_{1}$ (with $e_{3}\in \mC$) to the right-hand side of the corresponding expression listed in the above table. In this situation, let $\mathsf{p}_{5}$ denote one of the $\nicefrac{1}{5}(1,1,4)$-singularities of $X$. We see that this point lies in $\Pi_\mP\cap \mP(p_{3},u_1,s_{1})$. We may verify that $\Pi_{\mP}\cap \mP(p_{3},u_1,s_{1})=\{p_{3}^{2}=u_{1}s_{1}\}$.
Since $X$ is general, we may assume that $\mathsf{p}_{5}$ satisfies $s_{1}=1$ and $u_{1}=p_{3}^{2}$. Then, by applying the automorphism of $\Pi_{\mP}$ described in Proposition~\ref{prop:autom}~(1), we may move $\mathsf{p}_{5}$ to the $s_1$-point. Therefore, we may assume that $\mathsf{p}_{5}$ is the $s_{1}$-point of $\Pi_{\mP}^{13}$. Under this assumption, it follows that $e_{3}=0$. Consequently, the variety $X$ defined by the sections given in the above table is also general.

\vspace{5pt}

\noindent{}%
\fcolorbox{black}{white}{\textbf{No.$\,$878}}

\vspace{3pt}

\noindent\textbf{$\bullet$ Sections for $X$} \renewcommand{\arraystretch}{1.4}
\[
\begin{array}{|c|c|}
\hline
\text{weight} & 2\\
\hline
\text{section}
&
\begin{array}{c}
p_{2}=a_{2}t_{245}^{2}\\
t_{1}=a_{1}t_{245}^{2}\\
t_{126}=a_{126}t_{245}^{2}\\
t_{136}=a_{136}t_{245}^{2}
\end{array}
\\
\hline
\end{array}
\]

\[
\renewcommand{\arraystretch}{1.4}
\begin{array}{|c|c|}
\hline
3 & 4\\
\hline
\begin{array}{c}
p_4=a_4 p_{1}+b_4 t_{135}+c_4 t_{245}^{3}\\
t_{2}=b_{2}p_{1}+c_{2}t_{135}+d_2 t_{245}^{3}\\
t_{124}=a_{124}p_{1}+b_{124}t_{135}+c_{124}t_{245}^{3}\\
t_{125}=a_{125}p_{1}+b_{125}t_{135}+c_{125}t_{245}^{3} 
\end{array}
&
\begin{array}{c}
\begin{aligned}
p_{3} &= a_{3}u_1\\
&\quad + b_3 p_1 t_{245}\\
&\quad + c_3 t_{135} t_{245}\\
&\quad + d_3 t_{245}^4
\end{aligned}\\[4pt]
\begin{aligned}
t_{123} &= a_{123}u_1\\
&\quad + b_{123} p_1 t_{245}\\
&\quad + c_{123} t_{135} t_{245}\\
&\quad + d_{123} t_{245}^4
\end{aligned}
\end{array}
\\
\hline
\end{array}
\]

\noindent\textbf{$\bullet$ Embedding of $X$}: 
\[
X\subset\mP(t_{245},p_{1},t_{135},s_{1},u_{1},u_{2},s_{2},s_{3})=\mP(1,3^{2},4^{2},5^{2},6).
\]

\noindent\vspace{3pt}

\noindent For a more general $X$, the two sections of weight $4$ is obtained by adding terms $e_{3}s_{1}$ and $e_{123}s_1$ (with $e_{3},e_{123} \in \mC$) to the right-hand sides of the corresponding expressions listed in the above table. In this situation, let $\mathsf{p}_{4}$ denote one of the $\nicefrac{1}{4}(1,1,3)$-singularities of $X$. We see that this point lies in $\Pi_\mP\cap\mP(p_{3},u_{1},t_{123},s_{1})$. We may verify that
\[
\Pi_{\mP}\cap \mP(p_{3},u_{1},t_{123},s_{1})=\{p_{3}^{2}=u_{1}s_{1}\}.
\]
Since $X$ is general, we may assume that $\mathsf{p}_{4}$ satisfies $s_{1}=1$ and $u_{1}=p_{3}^{2}$. Then, by applying the automorphism of $\Pi_{\mP}$ described in Proposition~\ref{prop:autom}~(1), we may move $\mathsf{p}_{4}$ to a point in $\mP(t_{123},s_{1})$ with $s_1=1$. Moreover, the automorphism as in Proposition \ref{prop:autom} (2) moves
a point $\mP(t_{123},s_{1})$ with $s_1=1$ to the $s_1$-point. Therefore, we may assume that $\mathsf{p}_{4}$ is the $s_{1}$-point of $\Pi_{\mP}^{13}$. Under this assumption, it follows that $e_{3}=e_{123}=0$. Consequently, the variety $X$ defined by the sections given in the above table is also general.

\vspace{5pt}

\noindent{}%
\fcolorbox{black}{white}{\textbf{No.$\,$1766}}

\vspace{3pt}

\noindent\textbf{$\bullet$ Sections for $X$} 

\noindent}

\[
\renewcommand{\arraystretch}{1.4}
\begin{array}{|c|c|c|}
\hline
\text{weight} & 1 & 2\\
\hline
\text{section}
&
p_{2}=a_{2}t_{245}
&
\begin{array}{c}
p_{4}=a_{4}p_{1}+b_{4}t_{245}^{2}\\
t_{1}=a_{1}p_{1}+b_{1}t_{245}^{2}\\
t_{126}=a_{126}p_{1}+b_{126}t_{245}^{2}\\
t_{136}=a_{136}p_{1}+b_{136}t_{245}^{2}
\end{array}
\\
\hline
\end{array}
\]

\[
\renewcommand{\arraystretch}{1.4}
\begin{array}{|c|c|}
\hline
3 & 4\\
\hline
\begin{array}{c}
\begin{aligned}
p_{3} &= a_{3}u_{1}+b_{3}t_{135}\\
&\quad +c_{3}t_{245}p_{1}+d_{3}t_{245}^{3}
\end{aligned}\\[4pt]
\begin{aligned}
t_{2} &= b_{2}u_{1}+c_{2}t_{135}\\
&\quad +d_{2}t_{245}p_{1}+e_{2}t_{245}^{3}
\end{aligned}\\[4pt]
\begin{aligned}
t_{124} &= a_{124}u_{1}+b_{124}t_{135}\\
&\quad +c_{124}t_{245}p_{1}+d_{124}t_{245}^{3}
\end{aligned}\\[4pt]
\begin{aligned}
t_{125} &= a_{125}u_{1}+b_{125}t_{135}\\
&\quad +c_{125}t_{245}p_{1}+d_{125}t_{245}^{3}
\end{aligned}
\end{array}
&
\begin{array}{c}
\begin{aligned}
t_{123} &= a_{123}u_{2}+b_{123}p_{1}^{2}+c_{123}t_{245}u_{1}\\
&\quad +d_{123}t_{245}t_{135}+e_{123}t_{245}^{2}p_{1}+f_{123}t_{245}^{4}
\end{aligned}
\end{array}
\\
\hline
\end{array}
\]\noindent\textbf{$\bullet$ Embedding of $X$}: 
\[
X\subset\mP(t_{245},p_{1},u_{1},s_{1},t_{135},u_{2},s_{2},s_{3})=\mP(1,2,3^{3},4^{2},5).
\]

\subsection{Lemma for $cA$-singularities}

The following lemma is expected to be well known for experts; however, since we needed a version that is convenient for use in this paper, we include a proof. 
Its formulation and proof are based on \cite[Thm.~3 and Cor.~4]{Mo} and \cite[Prop.~4.6]{Pa}. 
It is used to determine the singularities of $\widetilde{X}$ along the flipped curve in the case of Type~Dp, and the singularities of $Y$ along $C$ in the case of Type~Div (as for the latter case, see also Proposition \ref{prop:simple cont}).

\begin{lem}
\label{lem:canonical form}
Let $g:=g(x)$ be an element of $\,\mC\{x\}$, where $x=(x_{1},x_{2},x_{3},x_{4})$. 
We assign weights $(r_{1},r_{2},1,1)$ to $(x_{1},x_{2},x_{3},x_{4})$. 
Assume the following conditions:

\begin{enumerate}[$(1)$]

\item $r_{1}\geq r_{2}\geq 0$.

\item The order of $g(x_{1},x_{2},x_{3},x_{4})$ is at least $3$.

\item Let $\,r:=r_{1}+r_{2}$. 
Then the weight of $g$ is $\,r$ (here, the weight of an element of $\,\mC\{x\}$ means the minimum of the weights of the monomials appearing in it), and the weight $r$ part $g_{r}$ of $\,g$ is a polynomial in $x_{3},x_{4}$ only
(then, from the definition of weights, $g_{r}$ is a homogeneous polynomial of degree $r$. 
Moreover, by the conditions $(2)$ and $(3)$, we have $r\geq 3$, and hence if $r_{2}=0$, then $r_{1}\geq 3$.)

\end{enumerate}

Then there exists a change of variables $x_{i}\mapsto \widetilde{x}_{i}\ (1\leq i\leq 4)$ that is congruent to the identity modulo $(x)^{2}$, such that $\,x_{1}x_{2}+g$ is transformed into a nonzero constant multiple of $\ \ \widetilde{x}_{1}\widetilde{x}_{2}+\widetilde{g}(\widetilde{x}_{3},\widetilde{x}_{4})$, where $\, \widetilde{g}$ has order $\,r$, and its order $\,r$ part is equal to a nonzero constant multiple of $g_{r}$. 
In particular, if $g_{r}$ is a product of $\,r$ pairwise coprime linear forms, then the singularity defined by $\,x_{1}x_{2}+g$ is an ordinary $cA_{r-1}$-singularity.

Moreover, assume that $g(x_1,0,0,0)$ is identically zero. 
Then, if necessary, by making an additional change of variables, the set $\{x_{2}=x_{3}=x_{4}=0\}$ is mapped to $\{\widetilde{x}_{2}=\widetilde{x}_{3}=\widetilde{x}_{4}=0\}$.
\end{lem}

\begin{proof}
We rewrite $g(x_{1},x_{2},x_{3},x_{4})$ in the form
\begin{equation}
x_{1}a+x_{2}b+c,
\label{eq:cA1}
\end{equation}
where $a,b\in\mC\{x_{1},x_{2},x_{3},x_{4}\}$ and $c\in\mC\{x_{3},x_{4}\}$. 
By the condition~(2), the orders of $a$ and $b$ are at least $2$. 
Let $m\ (m\geq 2)$ be the minimum of the orders of $a$ and $b$. 
By the condition~(3), we have
\begin{equation}
\text{the weights of }a,b\text{ are greater than }r_{2},r_{1},\text{ respectively}.
\label{eq:abwt}
\end{equation}

We rewrite the sum of $x_{1}x_{2}$ and \eqref{eq:cA1} as
\begin{equation}
(x_{1}+b)(x_{2}+a)-ab+c.
\label{eq:cA1'}
\end{equation}
We then make a change of variables $x_{1}'=x_{1}+b$, $x_{2}'=x_{2}+a$, and write this as $x_{1}'x_{2}'+g'(x_{1}',x_{2}',x_{3},x_{4})$. 
We check that $x_{1}',x_{2}',x_{3},x_{4},g'$ satisfy the conditions $(1)$--$(3)$ of the lemma. 
As in \eqref{eq:cA1}, we write $g'=x_{1}'a'+x_{2}'b'+c'$. 
Since the order of $ab$ is at least $2m$, the orders of $a',b'$ are at least $2m-1$. 
Hence the orders of $a',b'$ are greater than those of $a,b$. 
Moreover, by \eqref{eq:abwt}, the weight of $ab$ is greater than $r$, so the weight of $g'$ is also $r$, and its weight-$r$ part coincides with that of $g$. 
Note that $x_{i}'\equiv x_{i}\pmod{(x)^{2}}\ (i=1,2)$.

Now we apply \cite[Thm.~3]{Mo} with $\varphi=x_{1}x_{2}+g$, $m=1$, and $b=2$ (all symbols on the left-hand side follow the notation of \cite[Thm.~3]{Mo}). 
Fix an integer $n$ satisfying the conditions of \cite[Thm.~3]{Mo} for $\varphi$ and $b=2$. 
By repeating the same procedure as above, we obtain an expression
\[
\widetilde{\varphi}:=\widetilde{x}_{1}\widetilde{x}_{2}+\widetilde{x}_{1}\widetilde{a}+\widetilde{x}_{2}\widetilde{b}+\widetilde{c}
\]
in which the orders of $\widetilde{a},\widetilde{b}$ are at least $n$. 
Let $\psi$ be the power series obtained by pulling back $\widetilde{x}_{1}\widetilde{x}_{2}+\widetilde{c}$ via the inverse of this change of variables. 
Then $\varphi\equiv\psi\pmod{(x)^{n}}$. 
Therefore, by \cite[Thm.~3]{Mo}, there exists a change of variables, congruent to the identity modulo $(x)^{2}$, that sends $\varphi$ to a nonzero constant multiple of $\psi$. 
Composing this with the change of variables sending $\varphi$ to $\widetilde{\varphi}$, we obtain the desired change of variables in the first assertion of the lemma.

Next, assume that $g(x_1,0,0,0)$ is identically zero. 
When $x_{2}=x_{3}=x_{4}=0$, we can write $\widetilde{x}_{i}=v_{i}x_{1}^{\alpha_{i}}\ (1\leq i\leq 4,\ \alpha_{i}\geq 2,\ v_{i}\in\mC\{x_{1}\})$. 
Here $v_{i}\ (i\geq 2)$ is either $0$ or a unit, and for $\widetilde{x}_{1}$ we have $\alpha_{1}=1$ and $v_{1}\in\mC\{x_{1}\}^{\times}$. 
Thus $\widetilde{x}_{i}=v_{i}\left(\frac{\widetilde{x}_{1}}{v_{1}}\right)^{\alpha_{i}}\ (i\geq 2)$. 
We then make a change of variables $\widehat{x}_{i}=\widetilde{x}_{i}-v_{i}\left(\frac{\widetilde{x}_{1}}{v_{1}}\right)^{\alpha_{i}}\ (i=3,4)$. 

Then
\[
\widetilde{x}_{1}\widetilde{x}_{2}+\widetilde{c}(\widetilde{x}_{3},\widetilde{x}_{4})
=\widetilde{x}_{1}\widetilde{x}_{2}+\widetilde{c}\left(\widehat{x}_{3}+v_{3}\left(\frac{\widetilde{x}_{1}}{v_{1}}\right)^{\alpha_{3}},\ \widehat{x}_{4}+v_{4}\left(\frac{\widetilde{x}_{1}}{v_{1}}\right)^{\alpha_{4}}\right).
\]
Expanding the right-hand side, we can write it as
\[
\widetilde{x}_{1}\widetilde{x}_{2}+\widetilde{c}(\widehat{x}_{3},\widehat{x}_{4})+\widetilde{x}_{1}m
\quad (m\in\mC\{x\}).
\]
Making a further change of variables $\widehat{x}_{2}=\widetilde{x}_{2}+m$, we obtain
\[
\widetilde{x}_{1}\widetilde{x}_{2}+\widetilde{c}(\widetilde{x}_{3},\widetilde{x}_{4})
=\widetilde{x}_{1}\widehat{x}_{2}+\widetilde{c}(\widehat{x}_{3},\widehat{x}_{4}).
\]
Setting $x_{2}=x_{3}=x_{4}=0$, we have $\widehat{x}_{3}=\widehat{x}_{4}=0$, and hence $\widetilde{x}_{1}\widehat{x}_{2}=0$. 
Since this holds identically under $x_{2}=x_{3}=x_{4}=0$ and $\widetilde{x}_{1}=v_{1}x_{1}$, we must have $\widehat{x}_{2}=0$.
\end{proof}

\subsection{Characterization of simple $(2,1)$-contractions}
It is possible to characterize simple $(2,1)$-contractions by the final assertion of  Lemma~\ref{lem:canonical form}; however, checking the conditions to apply it requires a somewhat large amount of computation. 
The following proposition gives a characterization of simple $(2,1)$-contractions that focuses on the source rather than the target of the contraction unlike Lemma ~\ref{lem:canonical form}. The amount of computation needed to check the conditions turns out to be relatively small.

\begin{prop}
\label{prop:simple cont}
Let $(Y,o)$ be an analytic germ of a normal $3$-fold, and let $f\colon X\to (Y,o)$ be a $K$-negative divisorial contraction that contracts a prime divisor $E$ to a curve $C$. 
Assume further that the following conditions are satisfied:

\begin{enumerate}[$(1)$]

\item The point $o$ is a Gorenstein terminal singularity (a cDV singularity) of $\,Y$.

\item The curve $\,C$ is smooth at $o$.

\item The variety $\,X$ is smooth outside $f^{-1}(o)$, and the fiber of $f$ over a point of $\,C$ other than $o$ is $\mP^{1}$.
\item $E$ is a normal surface (actually, $E$ is already Cohen-Macaulay by \cite[Cor.~5.25]{KoMo2}).
\item The fiber $f^{-1}(o)$ contains exactly one singular point $p$ and it is a $\nicefrac{1}{r}(1,1,r-1)$-singularity for some $r\geq 2$. Moreover, $f^{-1}(o)$ consists of $\,r$ copies of $\ \mP^{1}$ that intersect only at $p$.

\item Let $\,h_{1}\colon X_{1}\to X$ be the weighted blow-up at $p$ with weights $\nicefrac{1}{r}(1,1,r-1)$, let $F$ be its exceptional divisor (which is isomorphic to $\mP(1,1,r-1)$), and let $E_{1}\subset X_{1}$ be the strict transform of $\,E$. 
Then $E_{1}|_{F}\in|\sO(r-1)|$ and $E_{1}|_{F}\simeq\mP^{1}$, and the strict transforms of the $r$ irreducible components of $f^{-1}(o)$ intersect $E_{1}|_{F}$ transversely at $r$ distinct points. 
(Note: when $r=2$, the condition $E_{1}|_{F}\simeq\mP^{1}$ follows from $E_{1}|_{F}\in|\sO(r-1)|$. 
When $r\geq3$, the condition $E_{1}|_{F}\simeq\mP^{1}$ is equivalent, under $E_{1}|_{F}\in|\sO(r-1)|$, to the condition that $E_{1}$ does not pass through the singular point of $F$.)
\end{enumerate}

Then the following statements hold:

\begin{enumerate}[$(a)$]

\item The divisor $E$ is smooth along $\,f^{-1}(o)$ outside $p$.

\item The morphism $f$ is a simple $(2,1)$-contraction.

\end{enumerate}
\end{prop}

\begin{rem*}
Though the main point is to prove $(b)$, the fact that $(a)$ follows from the other conditions is also useful, since it reduces the number of conditions that need to be checked when applying the proposition.
\end{rem*}

\begin{proof}
The main part of the proof consists in constructing an explicit factorization of $f\colon X\to Y$ within the framework of \cite[Thm.~3.3]{ChH}.

First we proceed with the conditions~(4) and (5) in mind. 
Write 
\[
f^{-1}(o)=\sum_{i=1}^{r} m_i \delta_i
\]
in $E$, where $\delta_i$ are irreducible components, and $m_i\in\mN$.
Then $-K_X\cdot\sum_{i=1}^{r} m_i\delta_i=1$, and since $r(-K_X\cdot\delta_i)\in\mN$ for all $i$, we obtain
\begin{equation}
\forall\,i,\ m_i=1,\quad -K_X\cdot\delta_i=\nicefrac{1}{r}.
\label{eq:1/r}
\end{equation}
This implies that $f^{-1}(o)$ is smooth outside $p$. 
Since $f^{-1}(o)$ is a Cartier divisor of $E$, the assertion~(a) follows.

Let $\delta_i'\,(1\leq i\leq r)$ be the strict transform of $\delta_i$ on $X_1$.
By the condition~(6), \eqref{eq:1/r}, and $K_{X_{1}}=h_{1}^{*}K_{X}+\nicefrac{1}{r}F$, we have $-K_{X_1}\cdot\delta_i'=0$. 
Hence $-K_{X_1}$ is nef over $Y$. Note that, from $K_X=f^{*}K_Y+E$ and \eqref{eq:1/r}, we obtain
\begin{equation}
	\forall\,i,\ E\cdot\delta_i=-\nicefrac{1}{r}.
	\label{eq:-1/r}
\end{equation}
By the condition~(1), $K_Y\sim0$ near $o$. Then, since $K_X=f^{*}K_Y+E$,  we have $E\sim K_X$ near $f^{-1}(o)$. 
Moreover, near $p$, we have $rK_X\sim0$, hence $E\sim(r-1)(-K_X)$. 
Together with the condition~(6), this gives
\begin{equation}
	E_{1}=h_{1}^{*}E-\nicefrac{r-1}{r}F.
	\label{eq:E1}
\end{equation}
Then, by the condition~(6), \eqref{eq:-1/r} and \eqref{eq:E1}, we have $E_1\cdot\delta_i'=-1$. 
Since $E_1$ is smooth along $\delta_i'$, we obtain
\begin{equation}
\sN_{\delta'_i/X_1}=\sO_{\mP^1}(-1)^{\oplus 2}.
\label{eq:-1-1}
\end{equation}
All constructions below are carried out relatively over $Y$. 
We can check that the divisor $-E_1-F$ is nef over $Y$ and numerically trivial only on $\delta_i'\ (1\leq i\leq r)$. 
Therefore, by the base point free theorem, $-E_1-F$ is semiample over $Y$, and a suitable multiple defines a birational morphism contracting exactly the curves $\delta_i'$. 
This is an $E_1$-flopping contraction, and by \eqref{eq:-1-1}, its flop $X_1\dashrightarrow X_2$ is an Atiyah flop.

Then the strict transform $E_2$ of $E_1$ is a $\mP^1$-bundle over $C$, and $X_2$ is smooth along $E_2$. 
We can check that the divisor $-K_{X_2}$ is nef over $Y$, and $-K_{X_2}+E_2$ is semiample over $Y$ and numerically trivial only on the fibers of $E_2\to C$. 
Let $X_2\to X_3$ be the contraction defined by a multiple of $-K_{X_2}+E_2$. 
The image of $E_2$ is a curve $C'$ mapped isomorphically onto $C$ by $X_3\to Y$, and $X_3$ is smooth along $C'$.

Let $F_3$ be the strict transform of $F$ on $X_3$. 
It is obtained by blowing up $F$ at the $r$ points $F\cap\delta_i'\ (1\leq i\leq r)$ and then contracting the $(-1)$-curve given by the strict transform of $E_1|_F$. 
Now we identify $E_1|_F$ with $\mP^1$ with homogeneous coordinates $z,w$, and regard the points $F\cap\delta_i'$ as the zero locus of a homogeneous polynomial $g(z,w)$ of degree $r$. 
Consider the surface
\[
F_3'=\{xy+g(z,w)=0\}\subset\mP(1,r-1,1,1),
\]
where $x,y,z,w$ have weights $(1,r-1,1,1)$. 
We shall check that $F_3'\simeq F_3$. 
This can be seen by tracking the toric VGIT determined by the matrix
\[
\begin{array}{@{} c *{5}{c} c @{}}
  & t & x & y & z & w & \\
  \multirow{2}{*}{$\Bigl(\!$}
   & 0  & 1  & r-1  & 1  & 1
   & \multirow{2}{*}{$\!\Bigr),$} \\
   & -1 & -1 & 0 & 0 & 0 &
\end{array}
\]
which gives the diagram
\[
\xymatrix{
& \widetilde{\mP}\ar[dr]\ar[dl]\\
\mP(1,r-1,1,1)&&\mP(r-1,1,1).
}
\]
Then the strict transform $F_3''$ of $F_3'$ under $\widetilde{\mP}\to\mP(1,r-1,1,1)$ is $\{xy+tg(z,w)=0\}$. 
We verify that $F_3''\to F_3'$ is the blow-up at the $x$-point, and $F_3''\to\mP(1,r-1,1)$ is the blow-up along $\{y=g(z,w)=0\}$, hence $F_3'\simeq F_3$.

Let $h_3\colon X_3\to Y$ be the morphism induced by the above construction. 
It is a $K$-negative divisorial contraction sending $F_3$ to $o$. 
For the weighted blow-up $h_1$, the weighted multiplicity of $E$ is $\nicefrac{r-1}{r}$ by \eqref{eq:E1} and the discrepancy is $\nicefrac{1}{r}$, hence
\[
K_{X_3}=h_3^{*}K_Y+F_3.
\]
On the other hand, by the classification in \cite[Thm.~1.1]{MoPr}, the contraction $f\colon X\to Y$ corresponds to the case $9^{\circ}:N\times(k1A)$ in \cite{MoPr}. 
Thus, as stated below \cite[Cor.~1.2]{MoPr}, a general member of $|-K_X|$ contains none of the components of $f^{-1}(o)$, has an $A_{r-1}$-singularity at $p$, and is isomorphic to its image in $Y$. 
Therefore, $Y$ has a $cA$-singularity at $o$. 
Then, by \cite[Thm.~1.2 and 1.3]{Ka} and the description of $F_3$, it holds that
\[
(Y,o)\simeq(\{xy+h(z,w)=0\},o),
\]
and $h_3\colon X_3\to Y$ is isomorphic to the weighted blow-up with respect to the weights $(1,r-1,1,1)$, where
$h(z,w)$ has weight $r$, and we may assume that its degree-$r$ part coincides with the above $g(z,w)$.

From now on, we identify $h_3$ with this weighted blow-up. 
Then
\[
X_3=(\{xy+t^{-r}h(tz,tw)=0\}\setminus\{x=y=z=w=0\})/\mC^{*},
\]
where $\mC^{*}$ acts on $(t,x,y,z,w)$ with weights $(-1,1,r-1,1,1)$.

We determine the defining equations of $C$. 
Since $C$ is smooth at $o$, it is defined by three functions whose linear parts involve three of the variables $x,y,z,w$. 
Suppose these variables are $y,z,w$. 
Then, by the implicit function theorem, $C$ is defined by three functions of the following type:
\begin{equation}
y=v_1 x^{\alpha_1},\quad z=v_2 x^{\alpha_2},\quad w=v_3 x^{\alpha_3},
\label{eq:xzw}
\end{equation}
where $v_i$ is either $0$ or a unit, and $\alpha_i\in\mN$ for $i=1,2,3$. 
Making the change of variables $z'=z-v_2 x^{\alpha_2}$ and $w'=w-v_3 x^{\alpha_3}$, we can write
\[
xy+h(z,w)=xy+h(z',w')+x\, m(x,y,z,w).
\]
Further setting $y'=y+m$, we obtain $xy+h(z,w)=xy'+h(z',w')$. 
Since $C$ satisfies $z'=w'=0$ and is contained in $Y$, it follows that $xy'=0$. 
Since this holds on $C$ independently of $x$, we must have $y'=0$. 
Thus $C$ is defined by $y'=z'=w'=0$.

It remains to consider the case where $x$ appears among the linear terms. 
For example, suppose they are $x,y,z$. 
Then, by the implicit function theorem, $C$ is defined by three functions of the following type:
\[
x=v_1 w^{\alpha_1},\quad y=v_2 w^{\alpha_2},\quad z=v_3 w^{\alpha_3},
\]
where $v_i$ and $\alpha_i$ have the same properties as before. 
The inverse images of these under $h_3$ satisfy
\[
tx=v_1(tw)^{\alpha_1},\quad t^{r-1}y=v_2(tw)^{\alpha_2},\quad tz=v_3(tw)^{\alpha_3}.
\]
From the first equation, the strict transform $C'$ of $C$ in $X_3$ satisfies
\begin{equation}
x=v_1 t^{\alpha_1-1} w^{\alpha_1}.
\label{eq:xv1}
\end{equation}
Note that $C'$ meets the $h_3$-exceptional divisor $\{t=0\}$ at the $x$-point. 
Therefore, in \eqref{eq:xv1}, we must have $v_1\neq0$ and $\alpha_1=1$. 
Then we may exchange the roles of $x$ and $w$, reducing to the case of \eqref{eq:xzw}. 
All other cases are reduced similarly to \eqref{eq:xzw}.

Therefore, by the uniqueness of the divisorial extraction centered at $C$ (see \cite[Thm.~4.9]{KoMo1}), we conclude that $f$ is a simple $(2,1)$-contraction.
\end{proof}

\end{document}